\documentclass[final,1p,times]{elsarticle}

\usepackage{amssymb}
\usepackage{bm}
\usepackage{amsthm}
\usepackage{amscd}
\usepackage{amsmath}
\usepackage{amsfonts}
\usepackage{amssymb}
\usepackage{graphicx}
\usepackage{color}
\newtheorem{theorem}{Theorem}[section]

\newtheorem{lemma}[theorem]{Lemma}

\usepackage{mathrsfs}
\usepackage{titletoc}
\numberwithin{equation}{section}

\newcommand{\R}{\mathbb{R}} 
\newcommand{\N}{\mathbb{N}} 
\newcommand{\eps}{\epsilon}

\newcommand{\be}{\begin{equation}}
	
	\newcommand{\ee}{\end{equation}}
\newcommand{\bea}{\begin{eqnarray}}
	\newcommand{\eea}{\end{eqnarray}}
\newcommand{\bna}{\begin{eqnarray*}}
	\newcommand{\ena}{\end{eqnarray*}}

\renewcommand{\le}{\left}
\newcommand{\ri}{\right}

\journal{arXiv}

\begin{document}
	
	\begin{frontmatter}
		\title{Singular metrics of constant negative $Q$-curvature in 
			Euclidean spaces}
		 
		\author[br1]{Tobias König}
\ead{koenig@mathematik.uni-frankfurt.de}
	\address[br1]{Institut für Mathematik, 
Goethe-Universität Frankfurt, 
Robert-Mayer-Str. 10, 60629 Frankfurt am Main, Germany}	 
		\author[ruc1,ruc2]{Yamin Wang}
		\ead{yaminwang@ruc.edu.cn}
		\address[ruc1]{Department of Mathematics,
			Renmin University of China, Beijing 100872, China}
		\address[ruc2]{Dipartimento di Matematica Guido Castelnuovo, Università  di Roma, La Sapienza,  Roma 00185, Italy}
	
		\begin{abstract}
			We study singular metrics of constant negative $Q$-curvature in the Euclidean space $\mathbb{R}^n$ for every $n \geq 1$. Precisely, we consider solutions to the problem 
\[
	(-\Delta)^{n/2}u=-e^{nu}\quad \text{on}\quad\mathbb{R}^{n}\backslash \{0\},
\]
under a finite volume condition  $\Lambda:=\int_{\mathbb{R}^n}e^{nu}dx$.  We  classify all  singular solutions  of the above equation based on their behavior at  infinity and zero. As a consequence of this, when $n=1,2$, we show that there  is  actually no singular solution.  Then adapting a variational technique, we obtain that for any  $n\geq 3$ and  $\Lambda>0$, the equation  admits solutions with prescribed asymptotic behavior.  These solutions correspond to metrics of constant negative $Q$-curvature, which are either smooth or have a singularity at the origin of logarithmic or polynomial type.  The present paper  complements previous works on the case of  positive  $Q$-curvature,  and also sharpens previous results in the nonsingular negative $Q$-curvature case. 
		\end{abstract}
		
		
	\end{frontmatter}
	
	\section{Introduction}
	In the last decades, there has been much analytic work on the study of prescribed $Q$-curvature equation,  which arises in conformal geometry, i.e.,
	\begin{equation}\label{ap1}
		(-\Delta)^{n/2}u=Ke^{nu}\quad \text{on}\quad\mathbb{R}^{n}, \quad n\geq 3
	\end{equation}	
for some function $K$. Geometrically, if a smooth function $u$ solves (\ref{ap1}), 	then the conformal metric $g_u:=e^{2u}|dx|^2$ has $Q$-curvature $K$, where $|dx|^2$ denotes the Euclidean metric on  $\mathbb{R}^n$ (see for example \cite{y1,cy,cy2,fe}). Most attention has been paid to investigating the problem when $K\geq 0$. Suppose $K=(n-1)!$, which equals the constant $Q$-curvature of the round sphere $\mathbb{S}^n$. Let $|\mathbb{S}^n|$ be the volume of unit sphere $\mathbb{S}^n$.  Assume further that $g_u$ has finite volume $\Lambda$.  It is well-known that the problem (\ref{ap1})  admits solutions 
	\begin{equation}\label{apb1}
	u_{\lambda,x_0}(x)=\ln \left(\frac{2\lambda}{1+\lambda^2|x-x_0|^2}\right)
\end{equation}
with $\Lambda=|\mathbb{S}^n|$,  where $\lambda>0$ and $x_0\in \mathbb{R}^n.$ The functions in the family (\ref{apb1}) are called standard solutions or spherical solutions. They can be obtained by the stereographic projection and the action of the M$\ddot{\text{o}}$bius group of conformal diffeomorphisms on $\mathbb{S}^n$. Let us first review some results related to (\ref{ap1}) with the above hypotheses in the literature. In the case $n=4$, Lin \cite{cs}  classified all solutions to (\ref{ap1})  and obtained the asymptotic behavior of $u$ at $\infty$. Then this was generalized by Martinazzi \cite{lma}  for every even $n\geq 4$. This is also an extension of  previous results in \cite{cy,r2,xx,r3}. Recently, Jin et al. \cite{TJ} derived the classification of solutions to (\ref{ap1}) when $n=3$, which was extended by Hyder \cite{AHY}  to all odd dimension $n\geq 3$.  These classifications have all been shown to be essentially optimal thanks to converse existence results for solutions of (\ref{ap1}). For these, we refer the reader to \cite{hmjde,A-C,lm,e2,hye,HH},  among others.  Particularly, 
 when  $n=1,2$,  there are also some pioneering works on this issue,  for example  \cite{chan,cl,pt,dm}. \vspace{0.2cm}

Let us now turn to the case of constant negative curvature $K$. We can assume with no loss of generality that $K\equiv -1$. In fact, suppose $u$ is a solution of (\ref{ap1}) with $K=-(n-1)!$. Then for any constant $c$, it is clear that $\tilde{u}:=u-c$ satisfies
$$	(-\Delta)^{n/2}\tilde{u}=-(n-1)! e^{nc}e^{n\tilde{u}}\quad \text{on}\quad\mathbb{R}^{n}.$$
That is to say, we can replace $-(n-1)!$ by any negative constant. The present paper is devoted to the study of solutions to the negative curvature problem for dimensions greater or equal to one. Consider 
	\begin{equation}\label{az1}
		(-\Delta)^{n/2}u=-e^{nu}\quad \text{on}\quad\mathbb{R}^{n}, \quad\quad \Lambda=
		\int_{\mathbb{R}^n}e^{nu}dx<\infty.
	\end{equation}	
For $n\geq 1$,  every solution to (\ref{az1})  is indeed smooth, as discussed by  \cite{dmc,AH, lma}. Contrary to (\ref{apb1}), there is no explicit entire solution to (\ref{az1}).  When $n\in \{1, 2\}$, by a simple application of maximum principle,  the problem (\ref{az1}) admits no solution.  Things are however different in higher dimensions.  For even $n\geq 4$,   Martinazzi \cite{lm2} obtained radially symmetric solutions to (\ref{az1}). Meanwhile,  he also described  the asymptotic behavior of all solutions $u$ of (\ref{az1}) at infinity.   Subsequent to that,  via a fixed-point method,  Hyder-Martinazzi  \cite{y4} obtained the  existence of solutions $u$  with prescribed asymptotic behavior.  Let $u$ be the solution to (\ref{az1}). Define 
\begin{equation}\label{aaz1}
	v(x):=-\frac{1}{\gamma_{n}}\int_{\mathbb{R}^{n}}\ln\left(\frac{1+|y|}{|x-y|}\right)
	e^{n u(y)}dy,
\end{equation}	
where $\gamma_{n}:=(n-1)!|\mathbb{S}^n|/2$ is  chosen such that
$(-\Delta)^{n/2}\ln\left(\frac{1}{|x|}\right)=\gamma_{n} \delta_0$
in the sense of distributions. We also write for simplicity $\Lambda_1:=(n-1)!|\mathbb{S}^n|.$  The above  conclusions can be  summarized as follows. \vspace{0.2cm}

\noindent \textbf{Theorem A }\,\cite{dmc, lm2}. \,\textit{For $n\in\{1,2\}$, there is no solution to (\ref{az1}). Let  $n\geq 4$ be even. Supposing that $u$ is a solution to (\ref{az1}), one has
	$$u(x)=v(x)+p(x),$$
where $p$  is a non-constant polynomial of even degree at most $n-2$. Moreover, there is
a closed set $Z\subset \mathbb{S}^{n-1}$ of Hausdorff dimension at most  $n-2$ such that for every compact subset $K\subset \mathbb{S}^{n-1}\backslash Z$,
	\begin{equation}\label{aac2x1}
			\lim_{t\rightarrow \infty}\frac{v(t\xi)}{\ln t}=\frac{\Lambda}{\gamma_n} \quad \quad \text{for}\quad \xi\in K.
	\end{equation}}

For the existence results, the following theorem shows that both the asymptotic behavior of $u$  and the constant $\Lambda$ can be  \vspace{0.2cm}  simultaneously prescribed.

\noindent \textbf{Theorem B }\,\cite{y4}. \,\textit{Let  $n\geq 4$ be even. Given a polynomial $p$ such that $\deg(p)\leq n-2$ and $x\cdot \nabla p(x)\rightarrow -\infty$ as $|x|\rightarrow \infty,$ for every $\Lambda>0$,  there exists a  solution $u\in C^{\infty}(\mathbb{R}^n)$ to (\ref{az1}) having the asymptotic behavior
	$$u(x)=\frac{\Lambda}{\gamma_n}\ln(|x|)+p(x)+o(\ln |x|) \quad \quad \text{as}\quad |x|\rightarrow \infty. $$}

\indent Since the problem (\ref{az1}) has been treated for even $n\geq 4$, we are particularly interested in the odd case.    As can be expected, in higher odd dimensions $n\geq 3$, the problem  (\ref{az1}) admits entire solutions, which was proven by  Hyder \cite{AH}. However,  there remains  a gap of  classification for these solutions.  That is exactly one of the purposes of current \vspace{0.2cm} work. 

In this paper, we focus on  a more general  problem compared to (\ref{az1}), namely, solutions to
\begin{equation}\label{az2}
	(-\Delta)^{n/2}u=-e^{nu}\quad \text{on}\quad\mathbb{R}^{n}\backslash \{0\}, \quad\quad \Lambda=
	\int_{\mathbb{R}^n}e^{nu}dx<\infty.
\end{equation}
It corresponds to a conformal metric with  $Q$-curvature equal to  negative one  everywhere except at the origin.  One may also give a further interpretation of (\ref{az2}) from the geometric point of view. Roughly speaking, when $n\geq 2$ and solution  $u$ behaves like $O(\ln|x|)$ near the origin, it will represent a conic constant $Q$-curvature metric on $\mathbb{S}^n$, due to the stereographic projection. If $u$ blows up faster at the origin, then zero should be viewed as an essential singularity of the metric. But in dimension $n=1$,   it has a different geometric \vspace{0.2cm} interpretation (see  \cite{dm}). 

To understand the definition of weak solutions to such equations, we give a brief introduction as follows. The intrigued reader can consult \cite{dmc,nezza,sl}  for much deeper explanations. Let $n\geq 1$, $\mathcal{S}(\mathbb{R}^n)$ be the Schwartz space of rapidly decreasing smooth functions and $\mathcal{S}'(\mathbb{R}^n)$ be its dual. Let $f\in \mathcal{S}'(\mathbb{R}^n)$ be a given tempered distribution. In the case $n$ is even,  if $u\in L^1_{loc}(\mathbb{R}^n)$ and for every $\psi\in C_c^{\infty}(\mathbb{R}^n\backslash \{0\})$
$$\int_{\mathbb{R}^n}u(-\Delta)^{n/2}\psi dx=\langle f, \psi \rangle,$$
then $u$ is a solution of 
\begin{equation}\label{xaz1}
	(-\Delta)^{n/2}u=f\quad \text{on}\quad\mathbb{R}^{n}\backslash \{0\}.
\end{equation}	
If  $n$ is odd, given  $\sigma>0$,  we consider the space
\begin{equation*}\label{xazf1}
	L_{\sigma}(\mathbb{R}^n):=\left\{u\in L^1_{loc}(\mathbb{R}^n):\; \int_{\mathbb{R}^n}\frac{|u(x)|}{1+|x|^{n+2\sigma}}dx<\infty\right\}.
\end{equation*}	
The norm in $L_{\sigma}$ is naturally given by
$$\|u\|_{L_{\sigma}}=\int_{\mathbb{R}^n}\frac{|u(x)|}{1+|x|^{n+2\sigma}}dx.$$
For $\psi\in \mathcal{S}(\mathbb{R}^n)$, denote $(-\Delta)^{n/2}\psi:=F^{-1}(|\xi|^nF\psi(\xi))$, where $F\psi(\xi)$ is the normalized Fourier transform given by
\begin{equation*}\label{xaz2}
	F\psi(\xi):=\frac{1}{(2\pi)^{\frac{n}{2}}}\int_{\mathbb{R}^n}\psi(x)e^{-ix \cdot \psi}dx.
\end{equation*}	
A weak solution to (\ref{xaz1}) can be defined in two possible ways. For one of the  definitions, write $(-\Delta)^{n/2}u:=(-\Delta)^{1/2}\circ (-\Delta)^{(n-1)/2}$ with the convention that $(-\Delta)^0$ is the identity. We say that $u$ is a solution of (\ref{xaz1}) if $\Delta^{(n-1)/2}u \in L_{\frac{1}{2}}(\mathbb{R}^n)$, $u\in W^{n-1,1}_{loc}(\mathbb{R}^n)$ and 
\begin{equation}\label{nnx1}
	\langle (-\Delta)^{(n-1)/2}u,\,(-\Delta)^{1/2}\psi \rangle:=\int_{\mathbb{R}^n}u(-\Delta)^{n/2}\psi dx=\langle f, \psi \rangle
\end{equation}	
for every $\psi\in \mathcal{S}(\mathbb{R}^n)$. The integral in (\ref{nnx1})  makes sense (see for example   \cite[Proposition 2.1]{AH}). 
For the other definition of weak solution to (\ref{xaz1}), we require $u\in L_{\frac{n}{2}}(\mathbb{R}^n)$ and 
\begin{equation*}\label{nnx2}
	\langle (-\Delta)^{n/2}u,\,\psi \rangle:=\int_{\mathbb{R}^n}u(-\Delta)^{n/2}\psi dx=\langle f, \psi \rangle
\end{equation*}	
for every $\psi\in \mathcal{S}(\mathbb{R}^n)$, where $\int_{\mathbb{R}^n}u(-\Delta)^{n/2}\psi dx$ is  well-defined.   In the regular case, the above two definitions are in fact equivalent, which was proven by \vspace{0.2cm} \cite{AH}.

Before stating our results, we review the by-now classical works in this direction. When the $Q$-curvature equals a positive constant, such singular problem has been studied  by  \cite{hmm,T-L,fma}.  Moreover, there are some papers devoted to similar equations on  Riemann surface, for instance \cite{ru2,jost1,ru1,mt}. The reader can also consult  \cite{xi1,xi2, xi3, xi4} for a nice survey of  other conformally invariant equations with isolated singularities.     This inspires us to explore singular solutions to (\ref{az2}).  As we shall see  in Lemma \ref{smooth 1}, every such solution is smooth away from the origin. If some conditions are imposed, it is at least H$\mathrm{\ddot{o}}$lder continuous  near the origin. We refer to \cite[Theorem 2.1]{hmm} for basic regularity results and similar  \vspace{0.2cm}  arguments.   

Our first result is the classification of all singular solutions to \eqref{az2} in terms of their asymptotic behavior at $0$ and $\infty$. 

\begin{theorem}\label{tr2}
	Let $n\geq 1$ and $u$ be a solution to (\ref{az2}).  Let $v(x)$ be defined as in (\ref{aaz1}). 
	Then there exist $\beta\in\mathbb{R}$ and polynomials $p$, $q$ of even degree at most $n-1$ bounded from above such that
	\begin{equation*}\label{ac2}
		u(x)=v(x)+p(x)+q\left(\frac{x}{|x|^2}\right)+\beta\ln(|x|),
	\end{equation*}
	where $v$ satisfies
	\begin{equation}\label{aac2x}
		\lim_{|x|\rightarrow \infty}\frac{v(x)}{\ln(|x|)}=\frac{\Lambda}{\gamma_n}\quad\quad \text{and}\quad\quad
		\lim_{|x|\rightarrow 0}\frac{v(x)}{\ln(|x|)}=0.
	\end{equation}
	Moreover, assume either (i) $\beta>-1$,  $p(x)\rightarrow-\infty$ as $ |x|\rightarrow\infty$, or (ii) $\beta\leq-1$,  $q(x)\rightarrow-\infty$ as $ |x|\rightarrow\infty.$ Then  for every multi-index  $k\in\mathbb{N}^n$ with $0<|k|\leq n-1$,
	\begin{equation*}\label{aakx3}
		\lim_{|x|\rightarrow\infty}D^k v(x)=0.
	\end{equation*}
\end{theorem}

Theorem \ref{tr2} improves some existing results in the literature in various respects. To start with, it complements the classification of singular solutions from \cite{T-L}, which holds for the case of positive constant curvature. As a byproduct,  Theorem \ref{tr2} also improves some of the above-mentioned results of  \cite{lm2} on regular solutions to the  problem \eqref{az1}  in at least two ways. Firstly, containing a special case ($q = 0$, $\beta = 0$), it covers the classification of regular solutions to \eqref{az1} in odd dimensions $n \geq 3$, which has been missing so far to our knowledge. Secondly, the upper-boundedness of $p$ is new even for the regular case. To see that in Theorem A, we shall give a clearer description about the set $Z$ which appears in \cite{lm2}. Write 
$$p(t\xi)=\sum_{i=0}^{d} a_i(\xi)t^i, \quad \quad d:=\deg(p)\leq n-2,$$
where $a_i$ is a homogeneous polynomial of degree $i$ for each $0\leq i\leq d$ or $a_i\equiv0$. Then the set $Z$ is given by
$$Z=\{\xi\in\mathbb{S}^{n-1}:\,a_d(\xi)=0\}. $$
One finds that the polynomial $p$ of Theorem A is not necessarily upper bounded. Take for instance $Z=\{x_1=0\}\cap \mathbb{S}^{n-1}$. Then Theorem A does not exclude the case of a non-upper bounded polynomial like $p(x)=-x_1^4+x_2^2$, while our Theorem \ref{tr2} does.
On the other hand, we note that our arguments do not yield the stronger assertion that $Z = \emptyset$. Indeed, this would mean that $p(x) \to -\infty$ as $|x| \to \infty$. Although we can exclude a lot of simple explicit examples like $p(x)= -x_1^4$ by arguing similarly  to \eqref{fa7},  we cannot in general rule out the case of $p$ being upper-bounded, but having non-empty \vspace{0.2cm} $Z$. 

As is already seen, our result (\ref{aac2x}) is stronger than (\ref{aac2x1}). It is worth mentioning that proving (\ref{aac2x1}) when $Z=\emptyset$ is relatively standard by now. Fortunately, we can deal with the general case in which $Z$ could be nonempty by using Campanato-space estimate  first introduced \vspace{0.2cm} in  \cite{hmm}.

Theorem \ref{tr2} states in particular that in dimensions $n=1,2$ the polynomials $p$ and $q$ must be constant. From this, it is not hard to deduce our next result about non-existence of singular solutions.

\begin{theorem}\label{th2x}
	For $n\in\{1,2\}$, there are no  solutions to (\ref{az2}).
\end{theorem}

Let us briefly describe the strategy used to prove Theorem \ref{tr2} and the main difficulties.  An important step in the proof is to employ Bôcher's theorem from \cite{bocher}, in conjunction with the Kelvin transform. Another difficulty, which turns to be substantially harder than that in the positive curvature case, is to prove upper-boundedness of polynomials. Since we do not have an easy lower bound for the auxiliary function $v$,  the conventional method used by Martinazzi \cite{lma} (see also \cite{AHY,TJ}) failed to work in our  situation.  Another major problem  is the description of the asymptotic behavior of the solution $u$ to (\ref{az2}). Since we also treat odd dimensions $n \geq 1$, we need to tackle some problems appearing in the nonlocal setting.  All of these  make the equation (\ref{az2}) a relatively challenging  object of \vspace{0.2cm} investigation.  

To overcome these obstacles, we adapt some parts of the strategy of \cite{hmm} to our setting.  But several key changes are made on account of the fact that the signs in certain inequalities (see  (\ref{f1a}) and (\ref{f1}) below) are reversed with respect to the positive curvature case.   For example, to prove upper-boundedness of $p$, we need to invoke an improved bound related to $v$,  which is different  from \cite{hmm}.  After refining our knowledge about $v$ and the polynomials $p$ and $q$, we can obtain the asymptotic behavior of $v$, and hence of $u$, at infinity \vspace{0.2cm} and zero.

In the regular case (\ref{az1}),  we give an improvement of Theorem \ref{tr2}  in dimensions 3 and 4. More precisely, 
\begin{theorem}\label{th2}
Let  $u$ be a solution to (\ref{az1}).  For $n\in\{3,4\}$,  one has $u(x)=v(x)+p(x)$, where 
the polynomial $p$ is quadratic and  $v$ has the asymptotic behavior
	\begin{equation*}\label{aak2}
		v(x)=\frac{\Lambda}{\gamma_n}\ln(|x|)+\tilde{c}_0+O(|x|^{-\tau}) \quad \quad \text{as}\quad |x|\rightarrow \infty
	\end{equation*}
	for some constant $\tilde{c}_0\in\mathbb{R}$ and every $0<\tau<1$, and $	D^kv(x)=O(|x|^{-k})$ as $|x|\rightarrow \infty$ for every integer $0<k\leq n-1$. 
\end{theorem}	

Clearly, Theorem \ref{tr2} gives the most general asymptotic behavior for solutions to (\ref{az2}) at zero and infinity. In the following  theorem, we construct solutions that have precisely the kind of behavior expected by  Theorem \ref{tr2}. 
\begin{theorem}\label{tr3s}
	Let $n\geq3$ and $p,q$ be polynomials of degree at most $n-1$. Assume that  one of the following holds: \vspace{0.1cm} \\
	\vspace{0.1cm} 
	(i) \, $\beta\in\mathbb{R}$, $p(x), \, q(x)\rightarrow-\infty$ as $ |x|\rightarrow\infty$;\\
	\vspace{0.1cm} 
	(ii) \, $\beta>-1$,  $q(x)$ is upper-bounded and  $p(x)\rightarrow-\infty$ as $ |x|\rightarrow\infty$;\\
	\vspace{0.1cm} 
	(iii) \, $\beta<-1$, $p(x)$ is upper-bounded and $q(x)\rightarrow-\infty$ as $ |x|\rightarrow\infty.$
	
	\noindent Then for every  $\Lambda>0$, 	there is a solution $u$ of (\ref{az2}) such that
	$$u(x)=\left(\frac{\Lambda}{\gamma_n}+\beta\right)\ln(|x|)+p(x)+o(\ln |x|) \quad \quad \text{as}\quad |x|\rightarrow \infty$$ 
	and  $$u(x)=\beta\ln(|x|)+q\left(\frac{x}{|x|^2}\right)+o(\ln |x|)\quad \quad \text{as}\quad |x|\rightarrow 0.$$ 
\end{theorem}	

In the proof of Theorem \ref{tr3s}, we proceed by  a standard variational argument. This method has been already exploited in several works, see for example, \cite{T-L,A-C,TJ,AH}.  Based on this, we are able to produce entire solutions of (\ref{az2})  with prescribed  	\vspace{0.2cm}  polynomials.

The remaining part of this paper is structured as follows: In Section 2,  we establish some key facts for the proof of Theorems \ref{tr2} and \ref{th2x}. In Section 3,  we complete the proof of these theorems.  In Section 4,  we prove Theorem  \ref{th2}.  Section 5 is devoted to the proof of Theorem \ref{tr3s}. Section 6 is an appendix and contains some auxiliary lemmas. \vspace{0.2cm} 

 In what follows, we  often  denote various
constants by the same generic letter $c$ or $C$. We also write  for simplicity $B_r(x):=\{y\in \mathbb{R}^n:\, |y-x|< r\}$ and $B_r:=B_r(0).$

\section{Preliminary steps}

In this section we make some essential preparations towards the proof of Theorems \ref{tr2} and \ref{th2x}. The first of the crucial ingredients we derive here consists in various bounds on the function $v$. The second one is a Bôcher-type decomposition of $u$ into its regular part $v$ and possibly singular harmonic remainder terms.  We present the two points into separate subsections below. 
	

\subsection{Bounds on \vspace{0.1cm}  $v$}
 To begin with,  we have the  lemma as follows.
\begin{lemma}\label{lea1c}
	Let $n\geq 1$, $u$ be a solution of (\ref{az2})  and $v(x)$ be defined as in (\ref{aaz1}). Then $v(x)\leq 0$ for $|x|< 1$, and 
\begin{equation}\label{f1a}
	v(x)\leq \frac{\Lambda}{\gamma_{n}}\ln (|x|) \quad \quad \quad \text{for}\quad  |x|\geq 1.
\end{equation}
\end{lemma}	
\proof  The proof is identical to \cite[Lemma 3.1] {hmm}, only by substituting $v$ by $-v$. $\hfill\Box$\\

Conversely, the following lower bound on $v$ holds.
\begin{lemma}\label{lea2c}
	Let $n\geq 1$.	For any $\epsilon>0$, there exists $R>0$ such that for all $|x|\geq R$,
	\begin{equation}\label{f1}
		v(x)\geq \left(\frac{\Lambda}{\gamma_{n}}-\epsilon\right)\ln (|x|)+\frac{1}{\gamma_{n}}\int_{B_1(x)}\ln (|x-y|)e^{nu(y)}dy.
	\end{equation}
	Moreover, one has $(-v)^+\in L^1(\mathbb{R}^n).$
\end{lemma}	
\proof Similar to \cite [Lemma 2.4] {cs} (also see \cite[Lemma 11]{lm2}), one can easily get (\ref{f1}). For the second assertion, using (\ref{f1}) and the Fubini theorem, we deduce that
\begin{equation}\label{f2c}
	\begin{split}
		\int_{\mathbb{R}^n\backslash B_1}(-v)^+ dx&\leq C \int_{\mathbb{R}^n}\int_{\mathbb{R}^n}\chi_{|x-y|<1} \ln \left(\frac{1}{|x-y|}\right)e^{nu(y)}dydx\\
		&= C\int_{\mathbb{R}^n} e^{nu(y)} \int_{B_1(y)}\ln \left( \frac{1}{|x-y|}\right)  dxdy\\
		&\leq C \int_{\mathbb{R}^n} e^{nu(y)} dy\\
		&\leq C,
	\end{split}
\end{equation}
where  $C>0$ is a constant. Moreover,  notice by (\ref{aaz1}) that
\begin{equation*}\label{fpt6}
	\begin{split}
		\int_{B_1}|v(x)|dx &\leq C \int_{B_1}\int_{\mathbb{R}^n} \ln\left(\frac{1+|y|}{|x-y|}\right)
		e^{n u(y)}dydx\\
		&\leq C\int_{\mathbb{R}^n} e^{n u(y)}  \int_{B_1}  \ln\left(\frac{1+|y|}{|x-y|}\right)dxdy\\
		&\leq C.
	\end{split}
\end{equation*}
This  together with (\ref{f2c}), we conclude $(-v)^+\in L^1(\mathbb{R}^n).$ Hence the lemma is confirmed. $\hfill\Box$\\

An easy, but useful consequence of Lemma \ref{lea2c} is that there is a set $S_0\subset \mathbb{R}^n$ of finite measure such that
\begin{equation}\label{set2}
	v(x)\geq -C\quad \quad \text{ on}\quad \quad  \mathbb{R}^n\backslash S_0.
\end{equation}

The following lemma can be seen as an averaged version of the lower bound (\ref{f1}) on $v$. The advantage is that the averaging permits to drop the hard-to-control second term on the right side of (\ref{f1}). This will turn out to be of decisive importance in our proof of Theorem \ref{tr2}. 

\begin{lemma} \label{lemma v lower bound integral}
	For any $q \geq 1$ and $\eps_1, \eps_2 >0$ there are constants $c = c(q, \eps_1, \eps_2)$ and $R= R(q, \eps_1, \eps_2)$ such that for all $0 < \rho \leq 1$, 
	$$ \frac{1}{\rho^{n+\eps_2}} \int_{B_\rho(x)} e^{qv(z)} dz \geq c |x|^{\left(\frac{\Lambda}{\gamma_n} - \eps_1\right) q} \qquad \text{ for\;  all }\quad |x| \geq R. $$
\end{lemma}
\proof Fix $0 < \rho \leq 1$, $q \geq 1$ and $\eps_1, \eps_2 >0$. Take $R´\geq 1$ so large that the lower bound (\ref{f1})  holds with $\eps = \eps_1$ and such that $q \|e^{nu}\|_{L^1(B_R^c)} < \eps_2$. Here $B_R^c$ denotes the complement of the ball of radius $R$ centered at $0$. Then using (\ref{f1}),  we have for all $|x| \geq 2R$ 
\begin{equation}\label{yue1}
	\begin{split}
		\int_{B_\rho(x)} e^{qv(z)} dz &\geq \int_{B_\rho(x)} \exp\left[q \left(\frac{\Lambda}{\gamma_n} - \eps_1\right) \ln |z| - q \int_{B_1(z)} \ln \frac{1}{|z-y|} e^{nu(y)} dy \right] \, dz \\
		&= \int_{B_\rho(x)} |z|^{\left(\frac{\Lambda}{\gamma_n} - \eps_1\right)q}  \exp\left[ - q \int_{B_1(z)} \ln \frac{1}{|z-y|} e^{nu(y)} dy \right] \, dz \\
		&\geq c |x|^{\left(\frac{\Lambda}{\gamma_n} - \eps_1\right)q} \int_{B_\rho(x)}   \exp\left[ - q \int_{B_1(z)} \ln \frac{1}{|z-y|} e^{nu(y)} dy \right] \, dz.
	\end{split}
\end{equation}
 Since $-\ln$ is convex, we may apply Jensen's inequality to the inner integral. Write
\[ g_z(y) = \begin{cases}
	|z-y|^{-q\|e^{nu}\|_{L^1(B_R^c)}} & \text{ if } |z-y| < 1, \\
	1 &\text{ if } |z-y| \geq 1,
\end{cases}
\]
and define on $B_R^c$ the probability measure 
\[ d \mu(y) := \frac{e^{nu(y)} dy}{\|e^{nu}\|_{L^1(B_R^c)}}. \]
Thanks to Jensen's inequality, one finds
\begin{equation}\label{yue2}
	\begin{split}
		- q \int_{B_1(z)} \ln \frac{1}{|z-y|} e^{nu(y)} dy &= \int_{B_R^c} -\ln g_z(y) \, d \mu (y) \geq -\ln \left( \int_{B_R^c} g_z(y) \, d \mu (y) \right).
	\end{split}
\end{equation}
Plugging (\ref{yue2}) into  (\ref{yue1}), we obtain
\begin{align*}
	\int_{B_\rho(x)} e^{qv(z)} dz &\geq c |x|^{\left(\frac{\Lambda}{\gamma_n} - \eps_1\right)q} \int_{B_\rho(x)} \frac{1}{\int_{B_R^c} g_z(y) \, d \mu (y)} dz \\
	&\geq c |x|^{\left(\frac{\Lambda}{\gamma_n} - \eps_1\right)q} \int_{B_\rho(x)} \frac{1}{\int_{B_R^c} 1 + |z-y|^{-\eps_2} \, d \mu (y)} dz.
\end{align*}
Applying Jensen's inequality a second time, for the $dz$-integral and the convex function $t \mapsto 1/t$, gives
\begin{align*}
	\int_{B_\rho(x)} e^{qv(z)} dz &\geq c |B_\rho|^2 |x|^{\left(\frac{\Lambda}{\gamma_n} - \eps_1\right)q} \frac{1}{\int_{B_\rho(x)} \int_{B_R^c} 1 + |z-y|^{-\eps_2} \, d \mu (y) \, dz} \\
	&\geq c \rho^{2n} |x|^{\left(\frac{\Lambda}{\gamma_n} - \eps_1\right)q} \frac{1}{ \int_{B_R^c} \int_{B_\rho(x)} 1 + |z-y|^{-\eps_2}  \, dz \, d \mu (y)} \\
	&\geq c \rho^{n+\eps_2} |x|^{\left(\frac{\Lambda}{\gamma_n} - \eps_1\right)q}.
\end{align*}
For the last step, we used Fubini and the fact that $\int_{B_\rho(x)} |z-y|^{-\eps_2} dz  \leq c \rho^{n-\eps_2}$. 
This completes the proof. 	$\hfill\Box$\\

\subsection{A decomposition \vspace{0.1cm} of $u$}

Here is the main result of this subsection. 

\begin{lemma}\label{lexm2}
	Let $n\geq 1$, $u$ be a solution of (\ref{az2}) and $v(x)$ be defined as in (\ref{aaz1}). Then 
	\begin{equation}\label{fs6}
		u(x)=v(x)+p(x)+q\left(\frac{x}{|x|^2}\right)+\beta \ln(|x|),
	\end{equation}
	where $\beta\in \mathbb{R}$,  $p$ and $q$ are polynomials of degree at most $n-1$.  
\end{lemma}	

Before proving  Lemma \ref{lexm2}, we would like to highlight some helpful tools. By the conformal invariance of \eqref{az2}, the inversion (or Kelvin transform) of $u$, defined by
\begin{equation}\label{btt8}
	\tilde{u}(x):=u\left(\frac{x}{|x|^2}\right)-2\ln|x|,
\end{equation}	
satisfies \eqref{az2} with volume $\tilde{\Lambda} =
\int_{\mathbb{R}^n}e^{n\tilde{u}}dx=\Lambda.$ We also define 
\begin{equation}\label{aaz2}
	\tilde{v}(x):=-\frac{1}{\gamma_{n}}\int_{\mathbb{R}^{n}}\ln\left(\frac{1+|y|}{|x-y|}\right)
	e^{n \tilde{u}(y)}dy.
\end{equation}	
As we shall see in the proof of Lemma \ref{lexm2}, we can in fact express
\begin{equation}\label{aaz3}
	\tilde{v}(x)=v\left(\frac{x}{|x|^2}\right)+\frac{\Lambda}{\gamma_{n}}\ln|x|.
\end{equation}

Our proof of Lemma \ref{lexm2} relies on the following generalized Bôcher theorem, which says that polyharmonic functions in the punctured unit ball can be  written as the sum of partial  derivatives  of the fundamental  solution and  a harmonic function near the origin.
We refer the reader to \cite[page 60]{bocher}, for more details about it.

\begin{lemma} \label{lea3} (Bôcher's theorem)\;
	Let $n$ be even and $\eta\in H^{\frac{n}{2}}(B_1\backslash \{0\})$ be a solution of	$(-\Delta)^{n/2}\eta=0$ on $B_1\backslash \{0\}\subset \mathbb{R}^{n}.$ 	If $\eta$ satisfies
	$$\int_{B_1} \eta(x)^+dx<\infty,$$
	where $\eta(x)^+=\max\{\eta(x), 0\}$, then $\eta$ takes  the form 
	$$\eta(x)=\sum_{|k|\leq n-1}c(k)D^{k}\ln\left(\frac{1}{|x|}\right) +p(x)\quad \;  \text{on} \quad B_1\backslash \{0\},$$
	where  $k=(k_1,k_2, \cdots,k_n)\in \mathbb{R}^{n}$  is a multi-index, $c(k)$ are constants and $p$ is a smooth solution of $(-\Delta)^{n/2}p=0$ on $B_1.$
\end{lemma}	

Equipped with these tools we can prove the main result of this subsection. 

\proof[Proof of Lemma \ref{lexm2}] 	The following  proof follows some ideas of \cite[Proposition 2.1, steps 1-2]{T-L}, with necessary modifications in our situation.  Observe that $v$ satisfies $(-\Delta)^{n/2}v=-e^{nu}$ on $\mathbb{R}^n$ in the sense of distributions.
In particular, if $n=1$, according to \cite[Lemma 5.6]{dmc},  we know that
$v(x)\in L_{\frac{1}{2}}(\mathbb{R})$ is well defined and
\begin{equation*}\label{tt1}
	(-\Delta)^{1/2}v=-e^{v} \quad \quad \text{in}\quad \quad \mathcal{S}'(\mathbb{R}). 
\end{equation*}
Thus the difference $\eta:=u-v$ solves $(-\Delta)^{n/2}\eta=0$ on $\mathbb{R}^n\backslash \{0\}.$  By Lemma \ref{lea2c}, we moreover have $(-v)^+\in L^1(\mathbb{R}^n)$, and thereby
\begin{equation}\label{f6}
	\begin{split}
		\int_{B_1}\eta(x)^+dx &\leq \int_{B_1}u(x)^+dx + \int_{B_1}(-v)(x)^+dx\\
		&\leq \int_{B_1} e^{nu(x)}dx+C\\
		&\leq C.
	\end{split}
\end{equation}
In what follows, we distinguish  two cases to \vspace{0.2cm}  discuss. 

\textbf{\textit{Case (i)}}:\,$n$ is  \vspace{0.1cm} even. 

Let $B_{r_i}\backslash \{0\}$ be punctured balls with $0<r_1<r_2$ for $i=1,2$. Due to  (\ref{f6})  and  Lemma \ref{lea3},  one has 
\begin{equation}\label{fsp2}
	\begin{split}
		\eta(x)&=p_i(x)+\sum_{0\leq |k|\leq n-1}c_i(k)D^{k}\ln\left(\frac{1}{|x|}\right)\\
		&=p_i(x)+c_{0,i} \ln(|x|)+\sum_{1 \leq  |k|\leq n-1}c_i(k)D^{k}\ln\left(\frac{1}{|x|}\right)\\
		&=p_i(x)+c_{0,i} \ln\left(\frac{1}{|x|}\right)+q_i\left(\frac{x}{|x|^2}\right)\quad   \text{on} \quad B_{r_i}\backslash \{0\},
	\end{split}
\end{equation}
where $c_{0,i}\in \mathbb{R}$, $p_i$ solves $(-\Delta)^{n/2}p_i=0$ on $B_{r_i}$ and   $q_i\left(\frac{x}{|x|^2}\right)$ is a polynomial of degree at most  $n-1$. In particular, $q_i$ has no \vspace{0.2cm}  constant term. 

We first claim that $p_1= p_2$ and $q_1= q_2$.  In fact, substituting $x$ by $\frac{x}{|x|^2}$ in (\ref{fsp2}), one finds
\begin{equation}\label{fsp3}
	\begin{split}
		\eta\left(\frac{x}{|x|^2}\right)=p_i\left(\frac{x}{|x|^2}\right)+c_{0,i} \ln(|x|)+q_i(x)\quad   \text{on} \quad  \mathbb{R}^n\backslash B_{\frac{1}{r_i}}.
	\end{split}
\end{equation}
For simplicity, we set 
$$P(x)=p_1(x)-p_2(x),\quad \quad C=c_{0,1}-c_{0,2},\quad \quad  Q(x)=q_1(x)-q_2(x).$$
By virtue  of (\ref{fsp3}), it holds that
\begin{equation}\label{fsp4}
	\begin{split}
		0=P\left(\frac{x}{|x|^2}\right)+C \ln(|x|)+Q(x)\quad   \text{on} \quad  \mathbb{R}^n\backslash B_{\frac{1}{r_i}}.
	\end{split}
\end{equation}
Owing to  $(-\Delta)^{n/2}P(x)=0$ on $B_{r_i}$, this implies that $P\left(\frac{x}{|x|^2}\right)$ is bounded on $\mathbb{R}^n\backslash B_{\frac{1}{r_i}}$.  Note also that $Q(x)$ is a polynomial without constant term. Hence, it follows from (\ref{fsp4}) that $Q(x)\equiv 0$. That is to say, 
\begin{equation*}\label{fsp5}
	\begin{split}
		0=P\left(x\right)+C \ln(|x|)\quad  \text{on} \quad   B_{r_i}\backslash \{0\}.
	\end{split}
\end{equation*}
Since $P(x)$ is bounded on $B_{r_i}$, we get $C=0$ and thereby $P(x)\equiv 0$. Thus, our claim \vspace{0.2cm} is true. 

Then replacing $B_{r_i}$ by consecutively large balls,  we obtain
\begin{equation*}\label{fu2}
	\begin{split}
		\eta(x)=q\left(\frac{x}{|x|^2}\right)+\beta\ln(|x|)+p(x)\quad   \text{on} \;  \mathbb{R}^n\backslash \{0\},
	\end{split}
\end{equation*}
where $\beta\in \mathbb{R}$, $p$ solves $(-\Delta)^{n/2}p=0$ on $\mathbb{R}^n$ and   $q\left(\frac{x}{|x|^2}\right)$ is  a polynomial  of degree at most $n-1$.  Therefore, 
\begin{equation}\label{fux2}
	u(x)=v(x)+\eta(x)=v(x)+p(x)+q\left(\frac{x}{|x|^2}\right)+\beta \ln(|x|) \quad   \text{on} \;  \mathbb{R}^n \backslash \{0\}.
\end{equation}

Next we prove that $p(x)$ is a polynomial  of degree at most $n-1$. To see this, we shall apply the Kelvin transform. Recall that  $\tilde{u}$ given by (\ref{btt8}) is a solution to (\ref{az2}). From (\ref{fux2}),  there exist some $\tilde{\beta}\in\mathbb{R}$ and a polynomial  $\tilde{q}$ of degree at most $n-1$ so that
\begin{equation}\label{fux3}
	\tilde{u}(x)=\tilde{v}(x)+\tilde{p}(x)+\tilde{q}\left(\frac{x}{|x|^2}\right)+\tilde{\beta }\ln(|x|) \quad   \text{on} \;  \mathbb{R}^n \backslash \{0\},
\end{equation}
where $(-\Delta)^{n/2}\tilde{p}=0$ on $\mathbb{R}^n$ and $\tilde{v}(x)$ is defined  as in (\ref{aaz2}). 
We infer that (\ref{aaz3}) is true. Indeed,   due to (\ref{btt8})  and (\ref{fux2}), 
\begin{equation}\label{bfs8}
	\tilde{u}(x)=v\left(\frac{x}{|x|^2}\right)+p\left(\frac{x}{|x|^2}\right)
	+q\left(x\right)-(2+\beta)\ln(|x|).
\end{equation}
Denote the term on the right hand of (\ref{aaz3}) by $$v_{\tilde{u}}(x):=v\left(\frac{x}{|x|^2}\right)+\frac{\Lambda}{\gamma_n}\ln|x|.$$
From (\ref{aaz1}) and (\ref{fux2}), it follows that
\begin{equation}\label{asn8}
	\begin{split}
		v_{\tilde{u}}(x)&=-\frac{1}{\gamma_n}\int_{\mathbb{R}^n}\ln \left(\frac{1+|y|}{|x|\left|\frac{x}{|x|^2}-y\right|}\right)e^{nu(y)}dy\\
		&=-\frac{1}{\gamma_n}\int_{\mathbb{R}^n}\ln \left(\frac{1+|y|}{|x|\left|\frac{x}{|x|^2}-y\right|}\right)|y|^{n\beta}e^{n\left(v(y)+p(y)+q\left(\frac{y}{|y|^2}\right)\right)}dy.
	\end{split}
\end{equation}	
With a change of variables, and using  $|x||y|\left|\frac{x}{|x|^2}-\frac{y}{|y|^2}\right|=|x-y|$, we deduce  from  (\ref{asn8})  that
\begin{equation}\label{awc8}
	\begin{split}
		v_{\tilde{u}}(x)&=-\frac{1}{\gamma_n}\int_{\mathbb{R}^n}\ln \left(\frac{1+|y|}{|x||y|\left|\frac{x}{|x|^2}-\frac{y}{|y|^2}\right|}\right)
		\frac{e^{n\left(v\left(\frac{y}{|y|^2}\right)+p\left(\frac{y}{|y|^2}\right)+q(y)\right)}}{|y|^{n\left(2+\beta\right)}}dy\\
		&=-\frac{1}{\gamma_n}\int_{\mathbb{R}^n}\ln \left(\frac{1+|y|}{|x-y|}\right)
		\frac{e^{n\left(v\left(\frac{y}{|y|^2}\right)+p\left(\frac{y}{|y|^2}\right)+q(y)\right)}}{|y|^{n\left(2+\beta\right)}}dy.
	\end{split}
\end{equation}	
Then inserting  (\ref{bfs8}) into (\ref{awc8})  leads to
\begin{equation*}\label{aws9}
	\begin{split}
		v_{\tilde{u}}(x)=-\frac{1}{\gamma_n}\int_{\mathbb{R}^n}\ln \left(\frac{1+|y|}{|x-y|}\right)
		e^{n\tilde{u}(y)}dy.
	\end{split}
\end{equation*}	
Consequently,  we obtain $\tilde{v}(x)=v_{\tilde{u}}(x)$ and then  (\ref{aaz3}) is \vspace{0.2cm} confirmed.

Next, employing  (\ref{aaz3}) and (\ref{fux3}), we derive
\begin{equation}\label{bes8}
	\tilde{u}(x)=v\left(\frac{x}{|x|^2}\right)+\tilde{p}(x)+\tilde{q}\left(\frac{x}{|x|^2}\right)+\left(\tilde{\beta}+\frac{\Lambda}{\gamma_n}\right)\ln(|x|).
\end{equation}	
Putting   (\ref{bfs8}) and (\ref{bes8})   together gives
\begin{equation}\label{bubz8}
	\tilde{p}(x)+\tilde{q}\left(\frac{x}{|x|^2}\right)+\left(\tilde{\beta}+\frac{\Lambda}{\gamma_n}\right)\ln(|x|)=p\left(\frac{x}{|x|^2}\right)
	+q\left(x\right)-(2+\beta)\ln(|x|).
\end{equation}	
where $q$ and $\tilde{q}$ are polynomials of degree at most $n-1$. Particularly, from  (\ref{bubz8}),
\begin{equation}\label{baa9}
	\tilde{u}(x)=\tilde{v}(x)+p\left(\frac{x}{|x|^2}\right)+q(x)
	-\left(\frac{\Lambda}{\gamma_n}+\beta+2\right)\ln(|x|).
\end{equation}		
and $\limsup_{|x|\rightarrow 0}|x|^{n-1}\tilde{q}\left(\frac{x}{|x|^2}\right)<\infty.$ This together with (\ref{bubz8}) again implies 
$$\limsup_{|x|\rightarrow 0}|x|^{n-1}p\left(\frac{x}{|x|^2}\right)=\limsup_{|x|\rightarrow \infty}\frac{p(x)}{|x|^{n-1}}<\infty.$$
It follows that  $p(x)\leq 1+|x|^{n-1}$ on $\mathbb{R}^n$. Since $(-\Delta)^{n/2}p=0$ on $\mathbb{R}^n$, by Lemma \ref{lem6}, we know that $p$ is a polynomial of degree at most $n-1$. Then one  can deduce from (\ref{bubz8}) that $\tilde{q}=p$ and $\tilde{p}=q$.  Therefore (\ref{fs6}) \vspace{0.2cm} is established. 

\textbf{\textit{Case (ii)}}:\,$n$ is \vspace{0.1cm} odd. 

In this case, one can see $(-\Delta)^{(n+1)/2}\eta=0$ on $\mathbb{R}^n\backslash \{0\}$. 
 Using this, together with  (\ref{f6}) and Lemma \ref{lea3},  we repeat the same arguments as above and then obtain (\ref{fs6}). Moreover,  $p, q$ are polynomials of degree at  most $n-1$. This finishes the proof of the lemma. $\hfill\Box$\\


\section{Proof of Theorems \ref{tr2} and \ref{th2x}}

We start by proving that the polynomials $p$ and $q$ found in Lemma \ref{lexm2} are necessarily upper-bounded. Notice that, contrary to the analogous proof in the positive curvature case, here we will need to use the averaged bound from Lemma \ref{lemma v lower bound integral}. 

\begin{lemma} \label{lea42} 
	For $n\geq 1$, there exists some constant $C$ such that 
	\begin{equation}\label{fg7}
		p(x),\; q(x)\leq C.
	\end{equation}
In particular, $p$ and $q$ must have even degree. If $n$ is even, then $\deg(p), \deg(q) \leq n-2$. 
\end{lemma}	
\proof We first remark that it suffices to prove upper boundedness of $p$ only. Indeed, as mentioned  before, the inversion $\tilde{u}(x) = u\left(\frac{x}{|x|^2}\right) - 2 \ln |x|$ also solves (\ref{az2}) and has the decomposition (\ref{baa9}). 
In other words, $q$ plays the same role for $\tilde{u}$ as $p$ plays for $u$. Hence repeating the following proof for $\tilde{u}$ will show upper-boundedness of $q$.  \vspace{0.2cm} 

So suppose by contradiction that $\sup_{x \in \R^n} p(x) = + \infty$. By \cite[Theorem 3.1]{gorin}, we can justify as in \cite[Lemma 11]{lma},  that there is $s >0$ and a sequence of points $x_k$ with $|x_k| \to \infty$ such that $p(x_k) \geq c|x_k|^s$. Since $\deg p \leq n-1$, we have $|\nabla p(x)| \leq C |x|^{n-2}$ for all $x \in \R^n$. If we define the radius $\rho_k := \frac{1}{2C} |x|^{2-n}$, we therefore have
\[ |p(x) - p(x_k)| \leq |x-x_k| \sup_{B_{\rho_k}(x_k)} |\nabla p| \leq 1 \qquad \text{ for all } x \in B_{\rho_k}(x_k). \]
In particular, since $|x_k| \to \infty$, this implies that 
\begin{equation} \label{lea43} 
	p(x) \geq 2c |x|^{s} \geq c |x_k|^{s} \qquad \text{ on } B_{\rho_k}(x_k) . 
\end{equation}
Note also that $q\left(\frac{x}{|x|^2}\right)$ is uniformly bounded for $|x| \geq 1$. Using this,  together with (\ref{lea43}) and Lemma \ref{lemma v lower bound integral}, we obtain for $q = n$ and some $\eps_1, \eps_2 > 0$, 
\begin{equation*}\label{faqc7}
	\begin{split}
		\int_{B_{\rho_k}(x_k)} e^{nu(y)} dy &= \int_{B_{\rho_k}(x_k)} e^{nv(y)} e^{np(y)} e^{nq\left(\frac{y}{|y|^2}\right)} |y|^{n \beta} dy \\
		&\geq c e^{n |x_k|^s} |x_k|^{n \beta} \int_{B_{\rho_k}(x_k)} e^{nv(y)} dy \\
		&\geq  c e^{n |x_k|^s} |x_k|^{n \beta} \rho_k^{(n+\eps_2)} |x_k|^{\left(\frac{\Lambda}{\gamma_n} - \eps_1\right)n} \\
		&= c \frac{e^{n |x_k|^s}}{|x_k|^{(n+\eps_2)(n-2) - \left(\frac{\Lambda}{\gamma_n} + \beta - \eps_1\right)n}} \to + \infty \quad \quad \text{as}\quad k \to \infty
	\end{split}	
\end{equation*}
due to $|x_k| \to \infty$. This clearly contradicts the fact that $\int_{\R^n} e^{nu(y)} dy < \infty$. Thus (\ref{fg7}) is proven, and the remaining assertions follow immediately.
$\hfill\Box$\\

With the help of Lemma \ref{lea42}, we can obtain an additional Hölder-type bound on $v$.

\begin{lemma}
	\label{lemma v hölder}
	Suppose that $\frac{\Lambda}{\gamma_n} + \beta > 0$. Then, as $|x| \to \infty$, for every $y \in B_1(x)$, we have 
	\[ |v(x) - v(y)| \leq o(1) \ln |x|, \]
	where $o(1)$ denotes a quantity which goes to zero as $|x| \to \infty$. 
\end{lemma} 
\proof The proof proceeds exactly like in \cite[Lemma 3.6]{hmm}. For the sake of
completeness, we give a sketch of the proof. First, the oscillation of $v$ is estimated in the form of a Campanato-type quantity, namely,
\begin{equation}\label{flk4}
	\sup_{\rho\in(0,4])}\frac{1}{\rho^{n+\frac{1}{\ln|x|}}}\int_{B_{\rho}(x)}\left|v(y)-\fint_{B_\rho(x)}v(z)dz\right|dy=o(1),
\end{equation}
where $o(1)\rightarrow 0$ as $|x|\rightarrow\infty$. Then, in a second step (which is in fact completely general and independent of the problem), the following Hölder-type bound  can deduced from (\ref{flk4}),
\begin{equation}\label{flk5}
	\sup_{x,y\in B_1(x), x\not=y}\frac{|v(x)-v(y)|}{|x-y|^{\frac{1}{\ln(1+|x|)}}}=o(1)\ln(1+|x|).
\end{equation}

The only slight difference in the proof of (\ref{flk4}) with respect to \cite{hmm} is that the bound on $\int_{B_r(x)} e^{nu(y)} dy$  can be derived in a simpler way in our case.  Indeed, using (\ref{f1a}), (\ref{fs6}) and Lemma \ref{lea42}, we estimate
\begin{align*}
	\int_{B_r(x)} e^{nu(y)} dy  \leq C \int_{B_r(x)} |y|^{n \left( \frac{\Lambda}{\gamma_n} + \beta\right)} dy \leq C r^n  |x|^{n \left( \frac{\Lambda}{\gamma_n} + \beta\right)}. 
\end{align*}
Choosing $r = 2 \sqrt \rho$ and employing \cite[Lemma A.2]{hmm},   one arrives at
\begin{equation}\label{flk6}
		\begin{split}
	\tilde{(I)}&:= \int_{B_\rho(x)} \fint_{B_\rho(x)} \int_{B_{2 \sqrt \rho}(x)} \left| \log\left( \frac{|z-\xi|}{|y - \xi|} \right)\right| e^{nu(\xi)} \, d\xi \, dz \, dy \\
	& \leq C \rho^n \int_{B_{2 \sqrt \rho}(x)} e^{nu(\xi)} \, d \xi \\
	&\leq C \rho^\frac{3n}{2} |x|^{n \left( \frac{\Lambda}{\gamma_n} + \beta\right)},
		\end{split}
\end{equation}
where  $\fint_{B_\rho(x)} := \frac{1}{|B_\rho|} \int_{B_\rho(x)} \sim \rho^{-n} \int_{B_\rho(x)}$. From here on, the proof can be continued as in \cite{hmm} to get both (\ref{flk4}) and (\ref{flk5}). We shall explain why this is available. As a matter of fact, in contrast to (\ref{flk6}), the estimate obtained in \cite{hmm} is 
\begin{equation}\label{flk7}
	 (I) \leq C \rho^{\frac{5n}{4} - \frac{\eps_2}{4}} |x|^{c_1} 
\end{equation}
with $c_1 := \frac{n(\alpha - \beta)}{2} + n \eps_1$. The properties of the exponents in (\ref{flk7})  needed in the following proof are  
\[ \frac{5n}{4} - \frac{\eps_2}{4} > n \quad \quad \text{and}\quad \quad c_1>0.\]
Comparing (\ref{flk6}) with (\ref{flk7}), we notice that the above two properties are fulfilled for our set of exponents, since $\frac{3n}{2} > n$ is clear, and $\frac{\Lambda}{\gamma_n} + \beta > 0$ is precisely our assumption.  Thus the desired result follows without difficulty. $\hfill\Box$\\

Now we are in a position to complete the last remaining part of the proof of Theorem \ref{tr2}. 

\begin{proof}
[Proof of Theorem \ref{tr2}]
In view of Lemmas \ref{lexm2} and \ref{lea42},  it only remains to prove the claimed asymptotic behavior of $v$. Let us begin by showing
	\begin{equation}\label{xse2}
		\lim_{|x|\rightarrow\infty}\frac{v(x)}{\ln(|x|)}=\frac{\Lambda}{\gamma_n}  \quad \text{and} \quad 	\lim_{|x|\rightarrow 0}\frac{v(x)}{\ln(|x|)}=0.
	\end{equation}
	
 With Lemmas \ref{lemma v lower bound integral} and \ref{lemma v hölder} at our disposition, we can proceed similarly to \cite[Theorem 1.1]{hmm},  but with reversed inequality signs. By Kelvin transform and the relation \eqref{aaz3} it suffices to prove $\lim_{|x|\rightarrow\infty}\frac{v(x)}{\ln(|x|)}=\frac{\Lambda}{\gamma_n} $. We split the proof into two cases according to the sign of \vspace{0.2cm}  $\frac{\Lambda}{\gamma_n} + \beta$. 

\textbf{\textit{Case (i)}}:\,\vspace{0.2cm}  $\frac{\Lambda}{\gamma_n} + \beta \leq 0$. 

In this case the proof is more elementary. Indeed, in this case we can estimate the term in \eqref{f1}. Applying \eqref{f1a}, \eqref{fs6} and Lemma \ref{lea42}, we derive
\begin{equation*}
	\begin{split}
	\int_{B_1(x)} \ln \frac{1}{|x-y|} e^{nu(y)} dy &\leq  \int_{B_1(x)} \ln \frac{1}{|x-y|} e^{nv(y)} |y|^{n\beta} dy\\
	& \leq \int_{B_1(x)} \ln \frac{1}{|x-y|} |y|^{n\left(\frac{\Lambda}{\gamma_n} + \beta\right)} dy \\
	&\leq C 
	\end{split}
\end{equation*}
since $\frac{\Lambda}{\gamma_n} + \beta \leq 0$. 
Hence the lemma follows immediately from \eqref{f1a} and \eqref{f1} with arbitrary small $\eps > 0$. \vspace{0.2cm} 

\textbf{\textit{Case (ii)}}:\,\vspace{0.2cm}  $\frac{\Lambda}{\gamma_n} + \beta > 0$. 

This case is less elementary and we shall need to invoke Lemma \ref{lemma v hölder}, which holds precisely in this situation. Assume by contradiction that $\lim_{|x| \to \infty} \frac{v(x)}{\ln |x|} \neq \frac{\Lambda}{\gamma_n}$. In view of the upper bound \eqref{f1a}, there must be $\delta > 0$ and a sequence $(x_k)$ with $|x_k| \to \infty$ such that 
\[ v(x_k) \leq \left(\frac{\Lambda}{\gamma_n} - 2 \delta\right) \ln |x_k|. \]
By Lemma \ref{lemma v hölder}, we obtain that 
\[ v(x) = v(x_k) + o(1) \ln |x_k|  \leq \left(\frac{\Lambda}{\gamma_n} -  \delta\right) \ln |x_k| \qquad \forall x \in B_1(x_k) \]
for all $k$ large enough. 
Integrating over $B_1(x_k)$, we thus get 
\[ \int_{B_1(x_k)} e^{v(y)} dy \leq |B_1| |x_k|^{\frac{\Lambda}{\gamma_n} -  \delta} \, . \]
But on the other hand, Lemma \ref{lemma v lower bound integral} (applied with $\rho = 1$, $q = 1$) asserts that 
\[ \int_{B_1(x_k)} e^{v(y)} dy \geq c |x_k|^{\frac{\Lambda}{\gamma_n} - \eps_1} \]
for all $\eps_1 > 0$. Choosing $\eps_1 < \delta$, we obtain a contradiction because $|x_k| \to \infty$. Thus \eqref{xse2} is \vspace{0.2cm} established. 

To complete the proof of Theorem \ref{tr2}, we are left to  justify
	\begin{equation}\label{awqc}
		\lim_{|x|\rightarrow\infty}D^kv(x)=0 \qquad \text{ for\; all } \quad k \in \N^n \quad \text{ with }\quad  0 < |k| \leq n-1
	\end{equation}
under the assumptions either (i) $\beta\leq-1$,  $q(x)\rightarrow-\infty$ as $ |x|\rightarrow\infty,$ or that (ii) $\beta>-1$,  $p(x)\rightarrow-\infty$ as $ |x|\rightarrow\infty$. \vspace{0.2cm} 

If  $\beta\leq -1$, we use the inversion of $u$ and (\ref{baa9}) to get 
\begin{equation}\label{vbt3}
	\begin{split}
		|D^k\tilde{v}(x)|\leq C\int_{\mathbb{R}^n}\frac{e^{n\tilde{u}(y)}}{|x-y|^k}dy \leq C\int_{\mathbb{R}^n}\frac{e^{n\left(\tilde{v}(y)+p\left(\frac{y}{|y|^2}\right)
				+q(y)\right)}}{|x-y|^k|y|^{n\left(2+\beta+\frac{\Lambda}{\gamma_{n}}\right)}}dy. 
	\end{split}
\end{equation}
Since $q(x)\rightarrow-\infty$ as $ |x|\rightarrow\infty,$ it is not hard to check that for some $\delta>0$,
\begin{equation}\label{vba3}
	\begin{split}
		\int_{\mathbb{R}^n\backslash B_{\delta}(x)}\frac{e^{n\left(\tilde{v}(y)+p\left(\frac{y}{|y|^2}\right)
				+q(y)\right)}}{|x-y|^k|y|^{n\left(2+\beta+\frac{\Lambda}{\gamma_{n}}\right)}}dy\rightarrow 0\quad \quad \text{as}\quad |x|\rightarrow\infty.	
	\end{split}
\end{equation}
On the other hand, since  $\tilde{u}(x)\leq C$  for $|x|$ large,  take $\delta=\delta(\epsilon)$ sufficiently small, then
\begin{equation}\label{vba4}
	\begin{split}
		\int_{B_{\delta}(x)}\frac{e^{n\tilde{u}(y)}}{|x-y|^k}dy\leq C	\int_{B_{\delta}(x)}\frac{1}{|x-y|^k}dy\leq \epsilon.
	\end{split}
\end{equation}
Together with (\ref{vbt3})-(\ref{vba4}),  one gets the desired result. 
 If $\beta>-1$, we infer that
 \begin{equation}\label{vbt2}
 	\begin{split}
 		|D^kv(x)|\leq C\int_{\mathbb{R}^n}\frac{e^{nu(y)}}{|x-y|^k}dy \leq C\int_{\mathbb{R}^n}\frac{|y|^{n\beta }e^{n\left(v(y)+p(y)+q\left(\frac{y}{|y|^2}\right)\right)}}{|x-y|^k}dy.
 	\end{split}
 \end{equation}
By arguing as above,  (\ref{awqc}) follows from (\ref{vbt2}) immediately. Thus the proof of  Theorem \ref{tr2} is complete.
\end{proof}

From Theorem \ref{tr2}, we can now easily deduce the claimed non-existence result for dimensions $n=1,2$. 

\begin{proof}
[Proof of Theorem \ref{th2x}]
We prove more than what is claimed, namely that in any dimension $n\geq 1$, the polynomials $p$ and $q$ cannot be constant at the same time. When $n = 1,2$, due to the degree bound from Theorem \ref{tr2}, we know that   $p$ and $q$ must be constant simultaneously. Because of this,  we conclude in particular that for $n=1,2$ no solution to \eqref{az2} can exist. \vspace{0.2cm} 

Indeed, if $p$ and $q$ are constant, by Theorem \ref{tr2} we have
\begin{align*}
\infty &> \int_{\R^n} e^{nu} dx \geq c \left(  \int_{\{ |x| \geq 1\}} |x|^{n\beta} e^{nv} dx +  \int_{\{ |x| < 1\}} |x|^{n\beta} e^{nv} dx \right)  \\
&=  c \left(  \int_{\{ |x| \geq 1\}} |x|^{n\left(\beta + \frac{\Lambda}{\gamma_n}\right) + o(1)} dx +  \int_{\{ |x| < 1\}} |x|^{n\beta +o(1)}dx \right)
\end{align*}
for some $c>0$ and $\beta \in \R$ (where $o(1)$ denotes a quantity that goes to zero as $|x| \to \infty$). Now the finiteness of the first integral implies $\beta \leq - 1 - \frac{\Lambda}{\gamma_n}$, while the finiteness of the second term gives $\beta \geq -1$. Since $\Lambda > 0$, this yields a contradiction. Hence $p$ and $q$ cannot be constant at the same time.  
\end{proof} \vspace{0.2cm} 

\noindent\textbf{Remark.} To conclude this section, we provide an alternative approach to prove upper-boundedness of $p$ and $q$, which does not rely on the sophisticated bounds on $v$ given in Lemma \ref{lemma v lower bound integral}. Beisdes, this shall be useful to the proof in our next theorem. Letting $n\in\{3,4\}$, we assert that
	\begin{equation}\label{f7}
		p(x),\; q(x)\leq C.
	\end{equation}
In fact, to prove $p(x)\leq C$, we let 
\begin{equation*}\label{fed6}
	p(x)=\sum_{i=1}^na_i(x_i-x_i^0)^2+c_0\quad\quad\text{for}\;\;n=3,4,
\end{equation*}
where $a_i, x_i^0$ and $c_0$ are constants.	We assume by contradiction that there is $a_{i_0}>0$ for some $1\leq i_0\leq n$.
Then  on the set 
\begin{equation}\label{fe6}
	S:=\{x\in \mathbb{R}^n:\, |x_i|\leq 1 \;\text{for}\; i \neq i_0 \},
\end{equation}
applying the Young inequality, we infer that
\begin{equation}\label{fc6}
	\begin{split}
		p(x)&\geq a_{i_0} (x_{i_0}-x_{i_0}^0)^2+C\\
		&\geq a_{i_0} x_{i_0}^2-2a_{i_0} x_{i_0}x_{i_0}^0+\tilde{C}\\
		&\geq (a_{i_0}-\epsilon) x_{i_0}^2-C_\epsilon{x_{i_0}^0}^2+\tilde{C}\\
		&\geq C_0|x|^2-\tilde{C}_0,
	\end{split}
\end{equation}
where  $\epsilon$ is chosen	small enough and the constant $C_0, \tilde{C}_0>0$. Moreover, one has  $|S|=\infty$.  On the other hand, observe that $q(x)$ is bounded in $B_1(0)$. By the inversion, we know that $q\left(\frac{x}{|x|^2}\right)$ is bounded in $\mathbb{R}^n\backslash B_1(0)$. From  (\ref{set2}) and (\ref{fc6}),  it then follows  that for $\beta\in \mathbb{R}$,
\begin{equation}\label{fa7}
	\begin{split}
		\int_{\mathbb{R}^n}e^{nu}dx&\geq\int_{S\backslash (S_0\cup B_1)}e^{n\left(v(x)+p(x)+q\left(\frac{x}{|x|^2}\right)+\beta \ln(|x|)\right)}dx\\
		&\geq C\int_{S\backslash (S_0\cup B_1)}|x|^{n\beta}e^{C_0 n |x|^2}dx\\
		&=\infty,
	\end{split}
\end{equation}
contradicting (\ref{az2}). Thus $a_i\leq 0$ for every $1\leq i\leq n$ and $p(x)\leq C$. \vspace{0.2cm} 

To prove $q(x)\leq C$, we shall make use of the  inversion again. Recalling (\ref{baa9}), 
since $(-\tilde{v})^+\in L^1(\mathbb{R}^n),$ one finds $\tilde{v}(x)\geq -C$ in $\mathbb{R}^n\backslash \tilde{S}_0$ on a finite measure set $\tilde{S}_0$. Set 
\begin{equation*}\label{fav7}
	q(x)=\sum_{i=1}^n\tilde{a}_i(x_i-\tilde{x}_i^0)^2+\tilde{c}_0\quad\quad\text{for}\;\;n=3,4.
\end{equation*}
Now proceed exactly as above with $p$ replaced by $q$.  We can  obtain (\ref{f7}) as asserted. Thus the proof is done.   $\hfill\Box$\\

\section{Proof of Theorem \ref{th2}}	
In this section,  for $n\in\{3,4\}$, we are going to improve the results in Theorem \ref{tr2}.  As an initial step, we adopt another method to  give the expression of  solutions $u$ of (\ref{az1}), in comparison to  Lemma \ref{lexm2}. 
\begin{lemma}\label{lem2c}
	Let $n\geq 3$, $u$ be a solution of (\ref{az1}) and $v(x)$ be defined as in (\ref{aaz1}). Then $p:= u-v$ is a nonconstant polynomial of even degree. Moreover, if $n$ is even,  $\deg(p)\leq n-2$; if $n$ is odd,  $\deg(p)\leq n-1$. 
\end{lemma}	
\proof Since the even case $n\geq 4$ has already been addressed in \cite[Lemma 12]{lm2}, we only discuss the odd case $n\geq 3$  here.  To verify this assertion, we can follow the arguments in \cite[Lemma 2.4]{AHY}, which is based on Fourier transform. We can also prove it in a more direct way. Since $(-\Delta)^{n/2}p=0$,  thanks to Lemma \ref{nn3}, one gets   $(-\Delta)^{(n-1)/2}p$ is a constant. Thus $(-\Delta)^{(n+1)/2}p\equiv0$. Then as in \cite[Theorem 9]{lma}, using  Lemma \ref{lea2c}, one can check that $|D^{n-1}p(x)|\rightarrow 0$ as $R\rightarrow \infty$ for $B_R\subset \mathbb{R}^n$.  Thus by (\ref{az1}), we know that $p$ is a polynomial of even degree at most $n-1$. 
We now claim that $p$ is not constant. In fact, due to (\ref{set2}), if $p$ is constant, we derive 
$$\int_{\mathbb{R}^n}e^{nu} dx=\int_{\mathbb{R}^n}e^{n(v(x)+p(x))} dx \geq C\int_{\mathbb{R}^n\backslash S_0}e^{nv(x)}dx=\infty,$$
contradicting (\ref{az1}).  This terminates the proof of the lemma. $\hfill\Box$\\

Relying on Lemma \ref{lem2c}, we can further obtain
\begin{lemma} \label{leas4c} 
	For $n\in\{3,4\}$, there exist constants $c_1\geq 0$ and $c_2>0$ such that 
	\begin{equation}\label{fma7c}
		p(x)\leq c_1-c_2 |x|^2.
	\end{equation}	
\end{lemma}	
\proof  Notice that $p$ must be  quadratic in the case $n=\{3,4\}$. Since the proofs of  (\ref{fma7c})  can be viewed as a special case of   (\ref{f7}),   we omit the details here.   $\hfill\Box$\\

\indent\textit{Completion of the proof of Theorem \ref{th2}.} Let $u$ be a solution of (\ref{az1}).  We shall examine the asymptotic behavior of $u$ near infinity, precisely, 
\begin{lemma} \label{lea8} 
	For $n=3,4$,  there exists some constant $\tilde{c}_0$ such that   for any $0<\tau<1$
	\begin{equation}\label{faas14}
		u(x)=\frac{\Lambda}{\gamma_{n}}\ln(|x|)+p(x)+\tilde{c}_0+O(|x|^{-\tau}) \quad \quad \text{as}\quad |x|\rightarrow \infty.
	\end{equation}
\end{lemma}		
\proof 	By Lemmas \ref{lea1c} and \ref{leas4c}, for $|x|$ large, we bound
\begin{equation*}\label{fl1}
	\begin{split}
		\int_{B_1(x)}\ln \left(\frac{1}{|x-y|}\right)e^{nu(y)}dy&= \int_{B_1(x)}\ln \left(\frac{1}{|x-y|}\right)e^{n(v(y)+p(y))}dy\\
		&\leq \int_{B_1(x)}\ln \left(\frac{1}{|x-y|}\right)|y|^{\frac{n\Lambda}{\gamma_n}}e^{-n (c_2 |x|^2-c_1)}dy\\
		&\leq C, 
	\end{split}
\end{equation*}
because $c_2>0$. This together with Lemma \ref{lea2c} yields
\begin{equation}\label{faa8d}
	\begin{split}
		\liminf_{ |x|\rightarrow \infty}\frac{v(x)}{\ln |x|}\geq \frac{\Lambda}{\gamma_{n}}-\epsilon.
	\end{split}
\end{equation}
Since $\epsilon$ is arbitary, from (\ref{faa8d}) and Lemma \ref{lea1c}, 
\begin{equation}\label{faa10}
	\lim_{|x|\rightarrow \infty}\frac{v(x)}{\ln|x|}=\frac{\Lambda}{\gamma_{n}}.
\end{equation}
Denote $\tilde{v}_0(x):=v\left(\frac{x}{|x|^2}\right)+\frac{\Lambda}{\gamma_{n}}\ln(|x|)$. It is not hard to  check that 
\begin{equation}\label{faa9}
	(-\Delta)^{n/2}\tilde{v}_0(x)=-|x|^{-n\left(2+\frac{\Lambda}{\gamma_{n}}\right)}e^{n \left(\tilde{v}_0(x)+p\left(\frac{x}{|x|^2}\right)\right)}\quad  \text{on} \quad \mathbb{R}^n\backslash\{0\}.
\end{equation}
Keep in mind  that $p\left(\frac{x}{|x|^2}\right)\rightarrow-\infty$ as $|x|\rightarrow 0$.  Moreover, using  (\ref{faa10}) directly leads to
\begin{equation*}\label{faas12}
	\lim_{|x|\rightarrow 0}\frac{\tilde{v}_0(x)}{\ln|x|}=0.
\end{equation*}
 Hence, the right hand side of (\ref{faa9}) belongs to $L^p(B_1)$ for any $p>1$. Applying the regularity theorems to (\ref{faa9}), we obtain that $\tilde{v}_0\in C^{0, \tau}(B_1)$ for any $0<\tau<1$. 	Thus, 
\begin{equation}\label{faas13}
	\tilde{v}_0(x)=\tilde{c}_0+O(|x|^{\tau})\quad \text{as}\quad |x|\rightarrow 0,
\end{equation}
where $\tilde{c}_0:=\tilde{v}_0(0)$. Then it follows from (\ref{faas13}) and Lemma \ref{lem2c} that
\begin{equation*}
	u(x)=\frac{\Lambda}{\gamma_{n}}\ln(|x|)+p(x)+\tilde{c}_0+O(|x|^{-\tau})
\end{equation*}
as $|x|\rightarrow \infty$. Consequently, the proof of the lemma is done. 	$\hfill\Box$\\

Now we are in position  to prove 
\begin{lemma}\label{lemma13}
	Let $0<k\leq n-1$ be an integer.  For $n\in\{3,4\}$, we have 
	\begin{equation}\label{aakxc}
		D^kv(x)=O(|x|^{-k})\quad \quad \text{as}\quad |x|\rightarrow \infty.
	\end{equation}
\end{lemma}	 
\proof  Recall that $v(x)$ is defined in (\ref{aaz1}). We differentiate under the integral sign to get 
\begin{equation}\label{vb2}
	\begin{split}
		|D^kv(x)|\leq C\int_{\mathbb{R}^n}\frac{e^{nu(y)}}{|x-y|^k}dy.	
	\end{split}
\end{equation}
Decompose $\mathbb{R}^n=A_1\cup A_2 \cup A_3$, where 
$$A_1=\{y:|y|\leq \dfrac{|x|}{2}\},\quad A_2=\{y:|x-y|\leq\dfrac{|x|}{2}\},\quad A_3=\mathbb{R}^n\backslash \{A_1\cup A_2 \}.$$
For $y\in A_1\cup A_3 $, we have $|x-y|\geq |x|/2,$ which leads to
\begin{equation}\label{af14}
	\begin{split}
		\int_{A_1\cup A_3} \frac{1}{|x-y|^\ell}e^{nu(y)}dy\leq  \frac{C}{|x|^{k}} \int_{A_1\cup A_3} e^{nu(y)}dy\leq  \frac{C}{|x|^{k}}
	\end{split}
\end{equation}
for $|x|$ large. When $y\in A_2 $, one finds $|x|/2\leq|y|\leq 3|x|/2$. As a consequence of Lemmas \ref{lem2c} and \ref{leas4c}
\begin{equation}\label{afk14}
	\begin{split}
		\int_{A_2}\frac{1}{|x-y|^\ell}e^{nu(y)}dy\leq  C	\int_{A_2}\frac{e^{n(v(y)-c_2|y|^2)}}{|x-y|^k}dy\leq \frac{C}{|x|^k}.
	\end{split}
\end{equation}
 Putting (\ref{af14})-(\ref{afk14}) together gives (\ref{aakxc}). This ends the proof of the lemma.   $\hfill\Box$\\

Combining Lemmas  \ref{lem2c}-\ref{lemma13}, we eventually conclude
Theorem \ref{th2}. $\hfill\Box$\\

\section{Proof of Theorem \ref{tr3s}}	

We proceed by using a variational argument. This method has been adopted in several works to find solutions (see for example \cite{A-C,TJ,AH}).  Denote $w_0(x)=\ln \left(\frac{2}{1+|x|^2}\right),$ where $w_0$ satisfies
$$	(-\Delta)^{n/2}w_0=(n-1)!e^{nw_0},\quad \quad \int_{\mathbb{R}^n} e^{nw_0}dx=|\mathbb{S}^n|.$$ 
Remember that $\Lambda_1:=(n-1)!|\mathbb{S}^n|.$ For $n\geq 3$, we would like to find a solution to (\ref{az2}) of the form
\begin{equation}	\label{nxz1}
	u=\tilde{\psi}+c_{\tilde{\psi}}-\frac{\Lambda}{\Lambda_1} w_0+\beta \ln(|x|)+p(x)+q\left(\frac{x}{|x|^2}\right)
\end{equation}	
for some function $\tilde{\psi}(x)$ satisfying $\lim_{|x|\rightarrow \infty}\tilde{\psi}(x)\in \mathbb{R}$.
The existence of solutions of the form (\ref{nxz1}) can be deduced from the following lemma. 
\begin{lemma}\label{lemff12}
Let  $n\geq 3$ and $p,q$ be polynomials of even degree at most $n-1$. Suppose one of the following holds: \vspace{0.1cm} \\
\vspace{0.1cm} 
(i) \, $\beta\in\mathbb{R}$, $p(x), \, q(x)\rightarrow-\infty$ as $ |x|\rightarrow\infty$;\\
\vspace{0.1cm} 
(ii) \, $\beta>-1$,  $q(x)$ is upper-bounded and  $p(x)\rightarrow-\infty$ as $ |x|\rightarrow\infty$;\\
\vspace{0.1cm} 
(iii) \, $\beta<-1$, $p(x)$ is upper-bounded and $q(x)\rightarrow-\infty$ as $ |x|\rightarrow\infty.$

\noindent Then for $\Lambda>0$, there exists a solution $\tilde{\psi}$ such that  $\tilde{\psi}(x)=o(\ln|x|)$ as $|x|\rightarrow \infty$ and 
	\begin{equation*}	\label{a16a}
		(-\Delta)^{n/2}\tilde{\psi}=-Ke^{n(\tilde{\psi}+c_{\tilde{\psi}})}+\frac{(n-1)!\Lambda}{\Lambda_1}e^{nw_0},
	\end{equation*}	
	where 
	\begin{equation}	\label{cc6}
		K(x):=|x|^{n\beta}e^{n\left(p(x)+q\left(\frac{x}{|x|^2}\right)-\frac{\Lambda}{\Lambda_1}w_0\right)},\quad \quad  c_{\tilde{\psi}}:=\frac{1}{n}\ln \frac{\Lambda}{\int_{\mathbb{R}^n}Ke^{n\tilde{\psi}}dx}.
	\end{equation}
In particular, under assumption (i) or (iii), one has  $\tilde{\psi}\in C^{\infty}(\mathbb{R}^n)$ and under assumption (ii), one has $\tilde{\psi}\in C^{\infty}(\mathbb{R}^n\backslash \{0\})\cap C^{0, \alpha}_{loc}(\mathbb{R}^n)$ for $0<\alpha<1$.
\end{lemma}
\proof   Let $g_0$ be the round metric on the sphere $\mathbb{S}^n=\{x\in \mathbb{R}^{n+1}:|x|=1\}.$  Denote $\pi: \mathbb{S}^n\backslash \{N\}\rightarrow \mathbb{R}^n$ the stereographic projection, where $N=(0,\cdots, 0,1)\in \mathbb{S}^n$  is the north pole, and $\tilde{K}=K\circ \pi,$ where $K$ is given by (\ref{cc6}).  In what follows, we consider the space given by Definition 6.4
\begin{equation}\label{af28}
	\begin{split}
		H^{\frac{n}{2}}(\mathbb{S}^n)=\{u\in L^{2}(\mathbb{S}^n): \; \|u\|_{H^{\frac{n}{2}}(\mathbb{S}^n)}^2:=\|u\|_{L^{2}(\mathbb{S}^n)}^2+\|(P_{g_0}^n)
		^{\frac{1}{2}}u \|_{L^{2}(\mathbb{S}^n)}^2<\infty\},
	\end{split}
\end{equation}	
where $P_{g_0}^n$ is the Paneitz operator of order $n$ with respect to $g_0$.  Define the functional $J$ on $H^{\frac{n}{2}}(\mathbb{S}^n)$ by
\begin{equation}	\label{aq17a}
	J(\psi)=\int_{\mathbb{S}^n}\left(\frac{1}{2}|(P_{g_0}^n)^{\frac{1}{2}}\psi|^{2}-\frac{(n-1)!\Lambda}{\Lambda_1} \psi\right)dg_0+ \frac{\Lambda}{n}\ln \left(\int_{\mathbb{S}^n}\tilde{K}e^{n(\psi-w_0\circ \pi)}dg_0\right).
\end{equation}
Using the Jensen inequality, one finds
\begin{equation}	\label{xx3}
	\frac{1}{|\mathbb{S}^n|} \int_{\mathbb{S}^n}\tilde{K}e^{n(\psi-w_0\circ \pi)}dg_0\geq \exp\left[\frac{1}{|\mathbb{S}^n|} \int_{\mathbb{S}^n} \ln \left(\tilde{K}e^{n(\psi-w_0\circ \pi)}\right)dg_0\right].
\end{equation}
Set $\bar{\psi}:=\frac{1}{|\mathbb{S}^n|}\int_{\mathbb{S}^n}\psi dg_0.$ 
Taking the logarithm of the  both sides of (\ref{xx3}) yields 
\begin{equation}	\label{xx4}
	\begin{split}
		\ln\left(\int_{\mathbb{S}^n}\tilde{K}e^{n(\psi-w_0\circ \pi)}dg_0\right)&\geq \frac{1}{|\mathbb{S}^n|} \int_{\mathbb{S}^n} \ln \left(\tilde{K}e^{n(\psi-w_0\circ \pi)}\right)dg_0+\ln |\mathbb{S}^n|\\
		&\geq 	\frac{1}{|\mathbb{S}^n|} \int_{\mathbb{S}^n} \ln \left(\tilde{K}e^{-nw_0\circ \pi}\right)dg_0+n\bar{\psi}+\ln |\mathbb{S}^n|,
	\end{split}
\end{equation}
where $C>0$ is a constant. Since $J(\psi+c)=J(\psi)$ for any constant $c$, we can take a minimizing sequence $\{\psi_k\}$ of (\ref{aq17a}) with  $\bar{\psi}_k:=\frac{1}{|\mathbb{S}^n|}\int_{\mathbb{S}^n}\psi_k dg_0=0$.  Moreover,   observe that
\begin{equation}	\label{xxv44}
	\begin{split}
		\int_{\mathbb{S}^n} \ln \left(\tilde{K}e^{-nw_0\circ \pi}\right)dg_0&\geq \int_{\mathbb{S}^n} \ln  \tilde{K}dg_0-C,
	\end{split}
\end{equation}
Under the assumption (i),  we can bound
\begin{equation}	\label{xxvc5}
	\begin{split}
\int_{\mathbb{S}^n}	\ln  \tilde{K}dg_0&= \int_{\mathbb{R}^n} \ln  \left(|x|^{n\beta}e^{n\left(p(x)+q\left(\frac{x}{|x|^2}\right)-\frac{\Lambda}{\Lambda_1}w_0\right)}\right)e^{nw_0}dx\\
&\leq  \left(\int_{B_1}+ \int_{\mathbb{R}^n\backslash B_1}\right) \left|\ln  \left(|x|^{n\beta}e^{n\left(p(x)+q\left(\frac{x}{|x|^2}\right)-\frac{\Lambda}{\Lambda_1}w_0\right)}\right)\right|e^{nw_0}dx\\
	&\leq C\int_{B_1}\frac{ |\ln |x||}{|x|^{2}}dx+C\int_{\mathbb{R}^n\backslash B_1}\frac{ \ln (1+|x|)}{|x|^{2(n-1)}}dx\\
	&\leq C.
	\end{split}
\end{equation}
For assumption (ii), similarly as in (\ref{xxvc5}), there holds  
\begin{equation}	\label{xxvc6}
	\begin{split}
		\int_{\mathbb{S}^n}	\ln  \tilde{K}dg_0\leq C\int_{B_1}\ln \left(\frac{1}{|x|}\right)dx+C\int_{\mathbb{R}^n\backslash B_1}\frac{ \ln (1+|x|)}{|x|^{2(n-1)}}dx\leq C.
	\end{split}
\end{equation}
Otherwise, given condition (iii), one can easily check
\begin{equation}	\label{xxvc7}
	\begin{split}
		\int_{\mathbb{S}^n}	\ln  \tilde{K}dg_0\leq C\int_{B_1}\frac{ |\ln |x||}{|x|^{2}}dx+C\int_{\mathbb{R}^n\backslash B_1}\frac{ \ln (1+|x|)^2}{|x|^{2n}}dx\leq C.
	\end{split}
\end{equation}
The  above estimates (\ref{xxvc5})-(\ref{xxvc7}) are due to $n\geq 3$.  Summarizing the three estimates,  together with (\ref{xxv44}), we derive
\begin{equation}	\label{xxv4}
	\begin{split}
		\int_{\mathbb{S}^n} \ln \left(\tilde{K}e^{-nw_0\circ \pi}\right)dg_0&\geq -C,
	\end{split}
\end{equation}
Combining  (\ref{aq17a}), (\ref{xx4}) and (\ref{xxv4}), we get
\begin{equation}\label{a25a}
	J(\psi_k)\geq \frac{1}{2}\|(P_{g_0}^n)^{\frac{1}{2}}\psi_k \|_{L^2(\mathbb{S}^n)}^2-\frac{\Lambda }{n}C.
\end{equation}
Thanks to  the Poincar$\mathrm{\acute{e}}$-type inequality,  one has 
\begin{equation}\label{aa25a}\|\psi_k\|_{L^2(\mathbb{S}^n)}\leq \|(P_{g_0}^n)^{\frac{1}{2}}\psi_k\|_{L^2(\mathbb{S}^n)}.
\end{equation}
Due to (\ref{a25a}) and (\ref{aa25a}), we obtain that $\{\psi_k\}$ is bounded in $H^{\frac{n}{2}}(\mathbb{S}^n)$. Then up to a subsequence, $\psi_k\rightharpoonup \psi_0$ weakly for some $\psi_0 $ in $ H^{\frac{n}{2}}(\mathbb{S}^n)$. This implies that
\begin{equation}\label{a26a}
	\lim_{k\rightarrow \infty} \int_{\mathbb{S}^n} \psi_kdg_0=\int_{\mathbb{S}^n}  \psi_0 dg_0,\quad \quad \|\psi_0\|_{H^{\frac{n}{2}}(\mathbb{S}^n)} \leq \liminf_{k\rightarrow \infty} \|\psi_k\|_{H^{\frac{n}{2}}(\mathbb{S}^n)}.
\end{equation}
Since the maps  $\psi \mapsto \psi$ and $\psi\mapsto e^\psi$ are compact from  $H^{\frac{n}{2}}(\mathbb{S}^n)$ to $L^p(\mathbb{S}^n)$  for any $p\geq1$ (see \cite[page 20]{T-L}), 
\begin{equation}\label{a27a}
	\lim_{k\rightarrow \infty} \ln \left(\int_{\mathbb{S}^n}\tilde{K}e^{n(\psi_k-w_0\circ \pi)}dg_0\right)= \ln \left(\int_{\mathbb{S}^n}\tilde{K}e^{n(\psi_0-w_0\circ \pi)}dg_0\right)
\end{equation}
and $	\lim_{k\rightarrow \infty}\|\psi_k\|_{L^2(\mathbb{S}^n)}=\|\psi_0\|_{L^2(\mathbb{S}^n)}.$
Then it follows from (\ref{af28}) and (\ref{a26a})  that
\begin{equation}\label{ac29a}
	\|(P_{g_0}^n)
	^{\frac{1}{2}}\psi_0 \|_{L^{2}(\mathbb{S}^n)}^2 \leq \liminf_{k\rightarrow \infty} \|(P_{g_0}^n)
	^{\frac{1}{2}}\psi_k \|_{L^{2}(\mathbb{S}^n)}^2.
\end{equation}
Taking into account of (\ref{aq17a}) and (\ref{a26a})-(\ref{ac29a}) gives 
\begin{equation}\label{b6a}
	J(\psi_0)\leq  \liminf_{k\rightarrow \infty} J(\psi_k)=m.
\end{equation}
Since $\psi_0\in H^{\frac{n}{2}}(\mathbb{S}^n), $  by (\ref{b6a}), we arrive at
$$J(\psi_0)=m:=\inf \left\{J(\psi):\, \psi\in H^{\frac{n}{2}}(\mathbb{S}^n) \right\}.$$
Hence $\psi_0$ is a minimizer of $J$, namely, $\psi_0\in H^{n}(\mathbb{S}^n)$ is a weak solution of
\begin{equation}\label{dxc1}
	P_{g_0}^n\psi_0=\frac{(n-1)!\Lambda}{\Lambda_1}-\dfrac{\Lambda \tilde{K}e^{n(\psi_0-w_0\circ \pi)} }{\int_{\mathbb{S}^n}\tilde{K} e^{n(\psi_0-w_0\circ \pi)}dg_0}.  
\end{equation}
Notice that $P_{g_0}^n\psi_0\in L^{p}(\mathbb{S}^n)$ for $p>1$.  Using the regularity theory \cite[Lemmas 2.5 and 2.6]{AH},  we know that $\psi_0\in C^{0}(\mathbb{S}^n)$.  Let $\tilde{\psi}=\psi_0\circ \pi^{-1}$. Since $\psi_0$ is continuous at the north pole, one has  $\tilde{\psi}=\psi_0\circ \pi^{-1}=o(\ln|x|)$ as $|x|\rightarrow \infty$. Moreover, from \cite{TB} and (\ref{dxc1}) , there holds 
\begin{equation*}\label{dx1}
	\begin{split}
		(-\Delta)^{n/2}\tilde{\psi}
		&=e^{nw_0}\left(\frac{(n-1)!\Lambda}{\Lambda_1}-\dfrac{\Lambda  \tilde{K}e^{n(\psi_0-w_0\circ \pi)}}{\int_{\mathbb{S}^n}\tilde{K} e^{n(\psi_0-w_0\circ \pi)}dg_0}\right)\circ \pi^{-1}\\
		&=\frac{(n-1)!\Lambda}{\Lambda_1}e^{nw_0}-Ke^{n(\tilde{\psi}+c_{\tilde{\psi}})}.
	\end{split}
\end{equation*}
Under assumption (i) or (iii), since $K(x)\in C^\infty(B_1)$, we can bootstrap and
use Schauder’s estimate to prove that $\tilde{\psi}\in C^{\infty}(\mathbb{R}^n)$. Under assumption (ii),  notice that $K(x)\in L^p(B_1)$ for $1<p<2$. Using the embedding theorem, we get $\tilde{\psi}\in W^{n,p}(B_1)\hookrightarrow C^{0,\alpha}(B_1)$ for some  $0<\alpha<1$. Thus the lemma is proven. $\hfill\Box$\\

\indent\textit{Completion of the proof of Theorem \ref{tr3s}.} Thanks to  Lemma \ref{lemff12}, setting $u=\tilde{\psi}+c_{\tilde{\psi}}-\frac{\Lambda}{\Lambda_1} w_0+\beta \ln(|x|)+p(x)+q\left(\frac{x}{|x|^2}\right),$ we obtain
\begin{equation*}\label{dxs1}
	\begin{split}
		(-\Delta)^{n/2}u&=(-\Delta)^{n/2} \tilde{\psi}-\frac{\Lambda}{\Lambda_1}(-\Delta)^{n/2}w_0\\
		&=-Ke^{n(\tilde{\psi}+c_{\tilde{\psi}})}= -e^{u}\quad\quad \quad  \text{on}\quad\quad \mathbb{R}^{n}\backslash \{0\}.
\end{split}
\end{equation*}	
In particular, one can see	$$\int_{\mathbb{R}^n}e^{nu}dx=\int_{\mathbb{R}^n}Ke^{n(\tilde{\psi}+c_{\tilde{\psi}})}dx=\Lambda.$$
Therefore,  (\ref{az2})  admits a singular solution $u$ when $n\geq 3$.  This  completes the proof of  the theorem.  $\hfill\Box$\\

\section{Appendix: Some auxiliary lemmas}

\begin{lemma}\label{smooth 1}
	Let $n\geq 3$ and $u$ be a solution of (\ref{az2}). Then $u\in C^{\infty}(\mathbb{R}^n\backslash\{0\}).$
\end{lemma}		
\proof  We split $-e^{nu}=f_1+f_2$, where $f_1,f_2\leq 0$,  $f_1\in L^1(\mathbb{R}^n)\cap L^{\infty}(\mathbb{R}^n)$ and $\|f_2\|_{L^1(\mathbb{R}^n)}<\frac{\gamma_{n}}{n}$. Denote
\begin{equation}\label{vn4}
	u_i:=\frac{1}{\gamma_{n}}\int_{\mathbb{R}^{n}}\ln\left(\frac{1+|y|}{|x-y|}\right)
	f_i(y)dy,\quad i=1,2,
\end{equation}
and $u_3:=u-u_1-u_2.$ Applying the regularity theorey  to $	(-\Delta)^{n/2}u_1=f_1$, we know  $u_1\in C^{n-1}(\mathbb{R}^n)$. Since $(-\Delta)^{n/2}u_3=0$ on $\mathbb{R}^n$,  one gets $u_3\in C^{\infty}(\mathbb{R}^n)$. On the other hand, for any $1<q\leq n$, it follows from the Jensen inequality and (\ref{vn4}) that
\begin{equation}\label{vn5}
	\begin{split}
		\int_{B_R(0)}e^{nqu_2}dx&= \int_{B_R(0)}\exp\left[\frac{nq}{\gamma_{n}}\int_{\mathbb{R}^{n}}\ln\left(\frac{1+|y|}{|x-y|}\right)f_2(y)dy\right]dx\\
		&\leq\int_{B_R(0)}\exp\left[\frac{nq\|f_2\|_{L^1(\mathbb{R}^n)}}{\gamma_{n}}\int_{\mathbb{R}^{n}}\ln\left(\frac{1}{|x-y|}\right)\frac{f_2(y)}{\|f_2\|_{L^1(\mathbb{R}^n)}}dy+C\right]dx\\
		&\leq C\int_{B_R(0)} \int_{\mathbb{R}^{n}}\frac{f_2(y)}{\|f_2\|_{L^1(\mathbb{R}^n)}}\exp\left[\frac{nq\|f_2\|_{L^1(\mathbb{R}^n)}}{\gamma_{n}}\ln\left(\frac{1}{|x-y|}\right)\right]dydx\\
		&\leq  C \int_{\mathbb{R}^{n}} \frac{f_2(y)}{\|f_2\|_{L^1(\mathbb{R}^n)}} \int_{B_R(0)}	\left(\frac{1}{|x-y|}\right)^{\frac{nq\|f_2\|_{L^1(\mathbb{R}^n)}}{\gamma_{n}}}dxdy,
	\end{split}
\end{equation}
where $C$ is a constant. Owing to $\|f_2\|_{L^1(\mathbb{R}^n)}<\frac{\gamma_{n}}{n}$, we get
\begin{equation}\label{vn6}
	\begin{split}
		\int_{B_R(0)}	\left(\frac{1}{|x-y|}\right)^{\frac{nq\|f_2\|_{L^1(\mathbb{R}^n)}}{\gamma_{n}}}dx \leq \int_{B_R(0)}\frac{1}{|x|^n}dx\leq C
	\end{split}
\end{equation}
for a constant $C$ depending on $R$. Putting (\ref{vn5}) and (\ref{vn6}) together, we have
\begin{equation*}
	\begin{split}
		\int_{B_R(0)}e^{nqu_2}dx\leq  \int_{\mathbb{R}^{n}} \frac{f_2(y)}{\|f_2\|_{L^1(\mathbb{R}^n)}} dy \leq C.
	\end{split}
\end{equation*}
This implies that  $e^{nu}\in L^q_{loc}(\mathbb{R}^n)$ for any $1<q\leq n$. Using elliptic regularity of  \cite{gt} and Schauder's estimates  to (\ref{az2}), we conclude $u\in C^{\infty}(\mathbb{R}^n\backslash\{0\})$. $\hfill\Box$\\	
	
A generalization of \cite[Lemma 15]{TJ} reads
\begin{lemma}\label{nn3}
If $w\in L_{\frac{1}{2}}(\mathbb{R}^n)$ satifies $(-\Delta)^{1/2}w=0$ on $\mathbb{R}^n$, then $w$ is a constant.
\end{lemma}		
\proof Since the proof is similar to \cite[Lemma 15]{TJ},  we omit the details here but refer the reader to \cite{TJ} for the case $n=3$.  $\hfill\Box$\\
	
The following lemma comes from \cite[Theorem 5]{lma}.
\begin{lemma}\label{lem6}
If $h$ satisfies $(-\Delta)^{m}h=0$ on $\mathbb{R}^n$ and $h(x)\leq C(1+|x|^{\ell})$ foe some $\ell \geq 2m-2 $. Then $h(x)$	is a polynomial of degree at most $\ell$.
\end{lemma}	 \vspace{0.3cm}

We shall  give a short introduction about the Paneitz operator and define $H^{\frac{n}{2}}(\mathbb{S}^n)$. The reader can consult  \cite{TB2,A-C2,cy,gjms}  for deeper explanations about them. \vspace{0.3cm} 

\noindent \textbf{Definition 6.4}\, Let $n\geq 3$ and $g_0$ be the round metric on the sphere $\mathbb{S}^n$.  Denote a orthonormal basis of $L^2(\mathbb{S}^n)$
$$\{Y^m_{\ell}\in C^{\infty}(\mathbb{S}^n):\, 1\leq m\leq N_{\ell},\, \ell=0,1,2,\cdots\},$$
where $Y^m_{\ell}$ is an eigenfunction of the Laplace–Beltrami operator $-\Delta_{g_0}$ corresponding to the eigenvalue $\lambda_{\ell}=\ell(\ell+n-1)$, and $N_{\ell}$ is the multiplicity of $\lambda_{\ell}$ (see \cite{st} for details). Given $u\in L^2(\mathbb{S}^n)$, we can write 
\begin{equation}\label{vnx6}
	u=\sum_{\ell=0}^{\infty}\sum_{m=1}^{N_{\ell}}u^m_{\ell}Y^m_{\ell},\quad \quad u^m_{\ell}\in\mathbb{R}
\end{equation}
and   $\|u\|_{L^2(\mathbb{S}^n)}^2=\sum_{\ell=0}^{\infty}\sum_{m=1}^{N_{\ell}}(u^m_{\ell})^2$. 
For  $u\in L^2(\mathbb{S}^n)$ with spherical harmonics expansion as in (\ref{vnx6}), we set
\begin{equation*}\label{af2c8d}
	\begin{split}
		H^{n}(\mathbb{S}^n)=\left\{u\in L^{2}(\mathbb{S}^n): \; \|u\|_{\dot{H}^n(\mathbb{S}^n)}<\infty\right\},
	\end{split}
\end{equation*}	
where $\|u\|_{\dot{H}^n(\mathbb{S}^n)}^2=\|P_{g_0}^nu\|_{L^2(\mathbb{S}^n)}^2$ and   the Paneitz operator on $H^{n}(\mathbb{S}^n)$  is defined by 
\begin{equation*}
	\label{be1}P_{g_0}^nu=\le\{\begin{array}{lll}
\sum_{\ell=0}^{\infty}\sum_{m=1}^{N_{\ell}}\prod^{\frac{n-2}{2}}_{k=0}(\lambda_{\ell}+k(n-k-1))
u^m_{\ell}Y^m_{\ell} \quad \quad \quad \quad \quad \quad  \;\; \; \text{for} \; \,n\; \, \text{even},\\[1.5ex]
\sum_{\ell=0}^{\infty}\sum_{m=1}^{N_{\ell}}
\left(\lambda_{\ell}+\left(\frac{n-1}{2}\right)^2\right)^{\frac{1}{2}}\prod^{\frac{n-3}{2}}_{k=0}(\lambda_{\ell}+k(n-k-1))
u^m_{\ell}Y^m_{\ell} \quad \text{for} \; \,n\; \, \text{odd.}
\end{array}\ri.
\end{equation*}
Note that $P_{g_0}^n$ is positive and $P_{g_0}^nu\in L^2(\mathbb{S}^n)$. Thus  its square root can be defined as 

\begin{equation*}
	\label{be2}(P_{g_0}^n)^{\frac{1}{2}}u=\le\{\begin{array}{lll}
	\sum_{\ell=0}^{\infty}\sum_{m=1}^{N_{\ell}}	\prod^{\frac{n-2}{2}}_{k=0}(\lambda_{\ell}+k(n-k-1))^{\frac{1}{2}}
		u^m_{\ell}Y^m_{\ell} \quad \quad \quad \quad \quad \;\;  \quad \; \text{for} \; \,n\; \, \text{even},\\[1.5ex]
		\sum_{\ell=0}^{\infty}\sum_{m=1}^{N_{\ell}}
		\left(\lambda_{\ell}+\left(\frac{n-1}{2}\right)^2\right)^{\frac{1}{4}}\prod^{\frac{n-3}{2}}_{k=0}(\lambda_{\ell}+k(n-k-1))^{\frac{1}{2}}
		u^m_{\ell}Y^m_{\ell} \quad \text{for} \; \,n\; \, \text{odd.}
	\end{array}\ri.
\end{equation*}
This operator is naturally well defined on the Sobolev space 
  for $n$ even,
\begin{equation*}\label{af2c9}
	\begin{split}
		H^{\frac{n}{2}}(\mathbb{S}^n)=\left\{u\in L^{2}(\mathbb{S}^n): \; \sum_{\ell=0}^{\infty}\sum_{m=1}^{N_{\ell}}\prod^{\frac{n-2}{2}}_{k=0}(\lambda_{\ell}+k(n-k-1))
		(u^m_{\ell})^2<\infty\right\}
	\end{split}
\end{equation*}	
and for $n$ odd
\begin{equation*}\label{af2c8}
	\begin{split}
		H^{\frac{n}{2}}(\mathbb{S}^n)=\left\{u\in L^{2}(\mathbb{S}^n): \; \sum_{\ell=0}^{\infty}\sum_{m=1}^{N_{\ell}}
		\left(\lambda_{\ell}+\left(\frac{n-1}{2}\right)^2\right)^{\frac{1}{2}}\prod^{\frac{n-3}{2}}_{k=0}(\lambda_{\ell}+k(n-k-1))
		(u^m_{\ell})^2<\infty\right\}
	\end{split}
\end{equation*}	
  endowed with the norm $\|u\|_{H^{\frac{n}{2}}(\mathbb{S}^n)}^2:=\|u\|_{L^{2}(\mathbb{S}^n)}^2+\|(P_{g_0}^n)
^{\frac{1}{2}}u \|_{L^{2}(\mathbb{S}^n)}^2.$

\vspace{1.0cm}

\noindent \textbf{Acknowledgements.} Tobias König acknowledges partial support  through ANR BLADE-JC ANR-18-CE40-002. Yamin Wang is supported by China Scholarship Council in 2021  and the Outstanding Innovative Talents Cultivation Funded Programs 2021 of Renmin University of China. She is very grateful to  Prof. Luca Martinazzi for many stimulating conversations during her visit to Università  di Roma, La Sapienza. She would also like to thank Prof. Ali Hyder for helpful discussions on this topic.

\end{document}